\documentclass[hidelinks,onefignum,onetabnum]{siamart220329}

\usepackage{lipsum}
\usepackage{amsfonts}
\usepackage{graphicx}
\usepackage{epstopdf}
\usepackage{algorithmic}

\ifpdf
  \DeclareGraphicsExtensions{.eps,.pdf,.png,.jpg}
\else
  \DeclareGraphicsExtensions{.eps}
\fi

\graphicspath{{./}{./figures/}}
\usepackage[noadjust]{cite} 

\usepackage{amsmath} 
\usepackage{amssymb}

\usepackage{mathrsfs}

\usepackage{tabularx}

\usepackage{verbatim}
\usepackage{float}

\usepackage{subcaption}

\usepackage{latexsym} 
\usepackage{pifont}

\usepackage{tikz}
\usetikzlibrary{fit,positioning,arrows,automata}

\usepackage{url}

\usepackage{capt-of}

\usepackage{enumerate}

\usepackage[]{algorithm}
\usepackage{algorithmic}

 \usepackage[normalem]{ulem}

\usepackage{bm}

\def\minwrt[#1]{\underset{#1}{\text{minimize }}}
\def\argminwrt[#1]{\underset{#1}{\argmin }}
\def\maxwrt[#1]{\underset{#1}{\text{maximize }}}
\def\supwrt[#1]{\underset{#1}{\text{sup }}}
\def\maxemphwrt[#1]{\underset{#1}{\text{\emph{maximize} }}}

\def\bK{{\bf K}}
\def\bM{{\bf M}}

\def\bU{{\bf U}}

\def\bC{{\bf C}}
\def\bT{{\bf T}}

\def\ett{{\bf 1}}

\def\ccE{{\mathcal{E}}}

\def\ccT{{\mathcal{T}}}

\def\ccP{{\mathcal{P}}}

\def\ccV{{\mathcal{V}}}
\def\ccG{{\mathcal{G}}}

\def\RE{{\mathbb{E}}}

\def\RR{{\mathbb{R}}}

\def\RRext{{\bar{\RR}}}

\newcommand{\trace}{\text{tr}}

\newcommand{\diag}{\mathrm{diag}}

\newcommand{\SpeciesMtx}{\mathfrak{R}}
\newcommand{\SpeciesMtxElem}{\mu}

\newcommand{\dom}{\text{dom}}
\newcommand{\interior}{\text{int}}
\newcommand{\epi}{\text{epi}}
\newcommand{\ri}{\text{ri}}
\DeclareMathOperator*{\argmin}{arg\,min}
\DeclareMathOperator*{\argmax}{arg\,max}

\newcommand{\newC}{\tilde{\bC}}

\newcommand{\indFun}{\iota}

\newsiamremark{remark}{Remark}

\newsiamremark{example}{Example}

\newsiamthm{assumption}{Assumption}

\headers{Manuscript}{Authors}

\title{Graph-structured tensor optimization for nonlinear density control and mean field games%
\thanks{Submitted to the editors DATE.
\funding{This work was supported by the Swedish Research Council (VR) under grant 2020-03454, KTH Digital Futures, the NSF under grant 1942523 and 2206576, the Knut and Alice Wallenberg foundation under grant KAW 2018.0349, and by the Wallenberg AI, Autonomous Systems and Software Program (WASP) funded by the Knut and Alice Wallenberg Foundation}}}

\author{%
Axel Ringh%
\thanks{Department of Mathematical Sciences, Chalmers University of Technology and the University of Gothenburg, Gothenburg, Sweden (\email{axelri@chalmers.se}).}
\and Isabel Haasler%
\thanks{Signal Processing Laboratory, LTS 4, \'{E}cole Polytechnique F\'{e}d\'{e}rale de Lausanne, Lausanne, Switzerland (\email{isabel.haasler@epfl.ch}).}
\and Yongxin Chen%
\thanks{School of Aerospace Engineering, Georgia Institute of Technology, Atlanta, GA, USA (\email{yongchen@gatech.edu}).}
\and Johan Karlsson%
\thanks{Division of Optimization and Systems Theory, Department of Mathematics, KTH Royal Institute of Technology, Stockholm, Sweden (\email{johan.karlsson@math.kth.se}).}
}

\usepackage{amsopn}

\begin{document}

\maketitle

% REQUIRED
\begin{abstract}
In this work we develop a numerical method for solving a type of convex graph-structured tensor optimization problems.
This type of problems, which can be seen as a generalization of multi-marginal optimal transport problems with graph-structured costs,
appear in many applications.
Examples are unbalanced optimal transport and multi-species potential mean field games, where the latter is a class of nonlinear density control problems.
The method we develop is based on coordinate ascent in a Lagrangian dual, and under mild assumptions we prove that the algorithm converges globally. Moreover, under a set of stricter assumptions, the algorithm converges R-linearly.
To perform the coordinate ascent steps one has to compute projections of the tensor, and doing so by brute force is in general not computationally feasible. Nevertheless, for certain graph structures it is possible to derive efficient methods for computing these projections, and here we specifically consider the graph structure that occurs in multi-species potential mean field games. We also illustrate the methodology on a numerical example from this problem class.
\end{abstract}

% REQUIRED

% REQUIRED

\section{Introduction}

A strong trend in many research fields is the study of large-scale systems consisting of components that are subsystems with specific characteristics.
Examples of such technological systems that are currently emerging include smart electric grids \cite{farhangi2010path}, and road networks with self-driving cars \cite{meyer2014road}.
There are also many such problems in biology, ecology, and social sciences, including, e.g., cell, animal, or human populations \cite{watts1998collective}. 
A major challenge is to understand and control the macroscopic behavior of such complex large-scale systems,
but since the number of agents in such systems
is often too large to model each agent individually,
the overall system is typically viewed as a flow or density control problem. 
In this setting, the aggregate state information of the agents is often described by a distribution or density function, and classical problems of this form include, e.g., network flow problems. More recently, there has been a large interest in control and estimation of densities,
including swarm control \cite{krishnan2018distributed, sinigaglia2021optimal}, modeling and control of epidemics \cite{lee2021controlling}, and covariance control in stochastic systems \cite{chen2016optimal}.
One key result is that certain density control problems of first-order integrators can be seen as optimal transport problems \cite{benamou2000computational}. This correspondence can be extended to general dynamics, and thus the optimal transport problem can be interpreted as a density control problem of agents (subsystems) with general dynamics \cite{hindawi2011mass, chen2017optimal, caluya2021wasserstein}.

While some 
density flow problems can be viewed as two-marginal optimal transport problems,
many problems involve using a time grid in order to model, e.g., congestion, instantaneous costs, or observations \cite{haasler2020optimal}. For such problems it is natural to use versions of the multi-marginal optimal transport problem.
The latter is an optimization problem where a nonnegative tensor is sought to minimize a linear cost subject to constraints on the marginals, where the marginals are projections of the tensor on specific modes.
When modeling the evloution of a system on a time grid, the marginals represent the distributions at different time points $j=1,\ldots, \ccT$.
More specifically, for such control problems with identical and indistinguishable agents,
the problem can be separate into $\ccT-1$ parts, where each part represents the evolution during time interval $[j,j+1]$ for $j=1,\ldots, \ccT-1$. The transition of the agents from time $j$ to time $j+1$ can thus be specified by the bimarginal projection of the tensor onto the joint two marginals $j$ and $j+1$, and thus this problem is a structured tensor problem with structure corresponding to a path graph (see, e.g., \cite{chen2018state,elvander2020multi,haasler2021multimarginal}).  
However, when the agents have heterogeneous dynamics or objectives, the distribution at a given time does not contain all necessary information about the past. Nevertheless, many problems of interest can instead be modeled by introducing additional dependencies between marginals.
For example, traffic flow problems with origin destination constraints can be formulated by introducing dependence between the initial and final node \cite{haasler2020optimal}, and Euler flow problems can be seen as a special case of this \cite{benamou2015bregman}.
By introducing an additional marginal representing different types of agents we can also formulate and solve multi-species dynamic flow problems and large multi-commodity problems \cite{haasler2021scalable}.
The resulting optimization problems are large-scale problems, but algorithms have been developed to solve this type of structured multi-marginal optimal transport problems  \cite{benamou2015bregman, haasler2020optimal, elvander2020multi, haasler2021multimarginal, haasler2021scalable, haasler2020multi, singh2020inference, altschuler2020polynomial, haasler2021control, fan2022complexity}. These extend Sinkhorn's method, developed for solving the bimarginal problem, in which an  entropy regularization is added to the cost function and an approximate solution is computed by using coordinate ascent in the dual problem \cite{cuturi2013sinkhorn, peyre2019computational}.
Interestingly, in the bimarginal problem the entropy regularization term can also be interpreted as introducing stochasticity in the dynamics of the subsystems, and the entropy-regularized problem can in fact be shown to be equivalent to the Schr\"odinger bridge problem \cite{leonard2012schrodinger, leonard2013schrodinger, chen2020stochastic}. This connection has also been extended to the 
multi-marginal setting \cite{haasler2021multimarginal}. Moreover, in this context it is also interesting to note that the algorithms developed to solve this type of structured multi-marginal problems are closely related to the unified propagation and scaling algorithm for inference in graphical models \cite{teh2002propagation}.

Many of the problems in the previous paragraph can be formulated as optimal transport problems with fixed marginals.
Nevertheless, in many situations it is also natural to consider problems where the marginals are not exactly known. A common strategy is then to penalize deviations from some given marginals \cite{georgiou2008metrics, piccoli2014generalized, karlsson2017generalized, chizat2018scaling, chizat2018unbalanced, liero2018optimal, BenCarDiNen19, beier2023unbalanced}. This is sometimes referred to as unbalanced optimal transport.
The cost functions associated with the non-fixed marginals are often convex,
but standard convex optimization methods in general do not scale to this type of large-scale problems.
In this paper, we develop a theoretical framework for a type of convex structured tensor optimization problems,
a generalization of graph-structured the multi-marginal optimal transport problems,
 along with numerical solution methods and convergence results for these. 
We also illustrate how this type of problems can be used to model and solve
multi-species potential mean field games, a generalization of the solution method for potential mean field games developed in \cite{BenCarDiNen19}.
An important observation is that the dual problem has a decomposable form, and can be efficiently solved using dual coordinate ascent \cite{karlsson2017generalized} (cf.~\cite{peyre2019computational}).
Moreover, the structure in these problems can be represented by a graph connecting the marginals, and by utilizing this graph we show how marginal and bimarginal projections of the tensor can be computed efficiently, thus alleviating the computational bottleneck of the algorithm.

The outline of the paper is as follows: in Section~\ref{sec:background} we introduce some background material on optimal transport and convex optimization.
The main results are presented in Section~\ref{sec:convex_graph_tensor_opt}, where we formulate the graph-structured tensor optimization problem of interest and present a primal-dual framework for solving it, together with a Sinkhorn-type algorithm for iteratively solving the dual problem. Conditions for convergence and R-linear convergence are also presented. Based on this, in Section \ref{sec:mean_field_games} we develop an
algorithm for solving multi-species potential mean field games.
This is done by casting the problem as a graph-structured tensor optimization problem and then specializing the general algorithm to the particular instance.
In the section, we also present a numerical example to illustrate the use and performance of the algorithm.
Finally, Section~\ref{sec:conclusions} contains some concluding remarks. Some proofs are deferred to the appendix for improved readability.
This paper builds on \cite{ringh2021efficient}, where we presented an algorithm, without proof of convergence, for the multi-species mean field game in a simplified setting (see also remark~\ref{rem:meanfieldgames}).

\section{Background}\label{sec:background}
This section presents background material, in particular on graph-structured multi-marginal optimal transport.
We also introduce some concepts from convex analysis and convex optimization that are needed in this work.

\subsection{The graph-structured multi-marginal optimal transport problem} \label{sect:OTgraph}
The optimal transport problem seeks a transport plan for how to move mass from an initial distribution to a target distribution with minimum cost. This topic has been extensively studied, see, e.g., the monograph \cite{villani2003topics} and references therein.
An extension of this problem is the multi-marginal optimal transport problem, 
in which a minimum-cost transport plan between several distributions is sought \cite{ruschendorf1995optimal, gangbo1998optimal, ruschendorf2002n, pass2015multi, benamou2015bregman, nenna2016numerical, elvander2020multi}.
In this work we consider the discrete case of the latter, where the marginal distributions are given by a finite set of $\ccT$ nonnegative vectors%
\footnote{To simplify the notation, we assume that all the marginals have the same number of elements, i.e., $\mu_j\in \RR^N$. This can easily be relaxed.}
$\mu_1, \ldots, \mu_\ccT \in \RR^N_+$.
The transport plan and the corresponding cost of moving mass are both represented by $\ccT$-mode tensors $\bM \in \RR^{N^\ccT}_+$ and $\bC \in \RR^{N^\ccT}$, respectively.
More precisely, the elements $\bM^{(i_1 \ldots i_\ccT)}$ and $\bC^{(i_1 \ldots i_\ccT)}$ are the transported mass and the cost of moving mass associated with the tuple $(i_1, \ldots, i_{\ccT})$, respectively.
The total cost of transport is therefore given by
\[
\langle\bC, \bM\rangle := \sum_{i_1, \ldots, i_\ccT} \bC^{(i_1 \ldots i_\ccT)} \bM^{(i_1 \ldots i_\ccT)}.
\]
Moreover, for $\bM$ to be a feasible transport plan, it must have the given distributions as its marginals. To this end, the marginal distributions of $\bM$ are given by the projections $P_j(\bM) \in \RR^N_+$, where elements of this vector are defined as
\begin{equation*}%\label{eq:projection}
(P_j(\bM))^{(i_j)} := \sum_{i_1, \ldots, i_{j-1}, i_{j+1}, i_\ccT} \bM^{(i_1 \ldots i_\ccT)},
\end{equation*}
and hence $\bM$ is feasible if $P_j(\bM) = \mu_j$ for $j = 1, \ldots, \ccT$.
A generalization of this optimization problem is to not necessarily impose marginal constraints on all projections $P_j(\bM)$, but only for an index set $\Gamma \subset \{1, \ldots, \ccT\}$.
The discrete multi-marginal optimal transport problem can thus be formulated as
\begin{subequations}\label{eq:discr_multimarginal}
\begin{align}
\minwrt[{\bM \in \RR^{N^{\ccT}}_+}] & \quad \langle \bC, \bM \rangle \\
\text{subject to } & \quad P_j(\bM) = \mu_j,\quad j \in \Gamma.
\end{align}
\end{subequations}

Problem \eqref{eq:discr_multimarginal} is a linear program, however solving it  can be computationally challenging due to the large number of variables. An approach for obtaining approximate solutions in the bimarginal case is to add a small entropy term to the cost function and solve the corresponding perturbed problem \cite{cuturi2013sinkhorn} (see also \cite{peyre2019computational}).
This perturbed problem can be solved by using the so-called Sinkhorn iterations.%
\footnote{In fact, the iterations have been discovered in different settings and therefore also have many different names; see, e.g., \cite{lamond1981bregman,chen2020stochastic}.}
The approach has been extended to the multi-marginal setting \cite{benamou2015bregman, nenna2016numerical, elvander2020multi}, however in this case it only partly alleviates the computational difficulty. More precisely, in the multi-marginal setting the entropy term is defined%
\footnote{In this work, we use the convention that $0 \cdot (\pm \infty) = (\pm \infty) \cdot 0 = 0$.} 
as
\[
D(\bM) := \sum_{i_1, \ldots, i_\ccT} \big( \bM^{(i_1 \ldots i_\ccT)}\log(\bM^{(i_1 \ldots i_\ccT)}) - \bM^{(i_1 \ldots i_\ccT)} + 1 \big),
\]
and the optimal solution to the perturbed problem can be shown to take the form
$
\bM = \bK \odot \bU
$,
see \cite{benamou2015bregman, elvander2020multi},
where  $\bK = \exp(- \bC/ \epsilon)$, $\odot$ denotes the elementwise product, and $\bU$ is the rank-one tensor
$
\bU^{(i_1 \ldots i_\ccT)} = \prod_{j \in \Gamma} u_j^{(i_j)},
$
i.e., $\bU = (\otimes_{j \in \Gamma} u_j) \otimes (\otimes_{j \in \{1, \ldots, \ccT \}\setminus \Gamma} \ett)$, where $\otimes$ denotes the tensor product and $\ett$ denotes a vector of ones.
In fact, the variables $u_j$ correspond to the logarithms of the Lagrangian dual variables in a relaxation of the entropy-regularized version of \eqref{eq:discr_multimarginal}.
Moreover, the (multi-marginal) Sinkhorn iterations iteratively update $u_j$ to match the given marginals:
\begin{equation*}%\label{eq:sinkhorn_multi}
u_j \leftarrow u_j \odot \mu_j \oslash P_j(\bK \odot \bU), \quad \text{ for } j \in \Gamma,
\end{equation*}
where $\oslash$ denotes elementwise division.
However, in the multi-marginal case, computing $P_j(\bK \odot \bU)$ is challenging since the number of terms in the sum grows exponentially with the number of marginals, and the latter is also reflected in complexity bounds for the algorithm \cite{lin2022complexity}. Nevertheless, in some cases when the underlying cost $\bC$ is structured the projections can be computed efficiently. In particular, this is the case for certain graph-structured costs.

To this end, let $\ccG=(\ccV, \ccE)$ be a connected graph with $\ccT=|\ccV|$ nodes, and consider the optimization problem
\begin{subequations}\label{eq:omt_multi_graph}
\begin{align}
\minwrt[\bM \in \RR_+^{N^{\ccT}}] & \quad \langle \bC, \bM \rangle + \epsilon D(\bM) \label{eq:omt_multi_graph_cost} \\  
\text{subject to} & \quad P_t (\bM) = \mu_t,\quad t \in \tilde \ccV,
\end{align} 
\end{subequations}
where $\tilde \ccV\subset \ccV$ is a set of vertices.
Moreover, consider cost tensor $\bC$ with the structure
\begin{equation}\label{eq:structured_cost}
\bC^{(i_1 \ldots i_\ccT)}=\sum_{(t_1,t_2)\in \ccE}C_{t_1,t_2}^{(i_{t_1}, i_{t_2})},
\end{equation}
where $C_{t_1,t_2} \in \RR^{N \times N}$, which in particular means that the linear cost term in \eqref{eq:omt_multi_graph_cost} takes the form
$\langle \bC, \bM \rangle= \sum_{(t_1,t_2)\in \ccE}\langle C_{t_1,t_2}, P_{t_1,t_2}(\bM)\rangle$.
Here $P_{t_1,t_2}(\bM)\in \RR_+^{N \times N}$ denotes the joint projection of the tensor $\bM$ on the two marginals $t_1$ and $t_2$, given by 
\begin{equation*}%\label{eq:bimarginal_projection}
(P_{t_1,t_2}(\bM))^{(i_{t_1}, i_{t_2})} := \sum_{\{i_1, \ldots, i_\ccT\} \setminus \{i_{t_1}, i_{t_2}\}} \bM^{(i_1 \ldots i_\ccT)}.
\end{equation*}
Problem \eqref{eq:omt_multi_graph} with a cost tensor structured according to \eqref{eq:structured_cost} is called a (entropy-regularized) graph-structured multi-marginal optimal transport problem \cite{haasler2021scalable, haasler2021multimarginal, fan2022complexity}. Moreover,
for many graph structures, the projections $P_{t}(\bM)$ and $P_{t_1,t_2}(\bM)$ can be efficiently computed, see, e.g, \cite{benamou2015bregman, haasler2020optimal, elvander2020multi, haasler2021multimarginal, haasler2021scalable, haasler2020multi, singh2020inference, altschuler2020polynomial, haasler2021control, fan2022complexity}, and hence the Sinkhorn iterations can be used to efficiently solve such problems.

\subsection{Convex analysis and optimization}
We need the following definitions and results from convex analysis and optimization. For extensive treatments of the topic, see, e.g., the monographs
\cite{rockafellar1970convex, bauschke2017convex}.
To this end, let $f : \RR^n \to \RRext:= \RR \cup\{\pm\infty\}$ be an extended real-valued function. The \emph{epigraph} of $f$ is defined as
$\epi(f) := \{ (x,\eta) \in \RR^n \times \RR \mid f(x) \leq \eta \}$, and
$f$ is called convex if $\epi(f) \subset \RR^{n+1}$ is a convex set. A function $f$ is \emph{lower-semicontinuous} if and only if $\epi(f)$ is closed \cite[Thm.~7.1]{rockafellar1970convex}. The \emph{effective domain} of $f$ is defined as $\dom(f) := \{ x \in \RR^n \mid f(x) < \infty \}$, and $f$ is called \emph{proper} if $f(x) > -\infty$ for all $x \in \RR^n$ and  $\dom(f) \neq \emptyset$.
A convex set $C$ is called \emph{polyhedral} if it can be written as the intersection of a finite number of closed half spaces. A convex function $f$ is called polyhedral if $\epi(f)$ is polyhedral.
The \emph{Fenchel conjugate} of a function $f$ is defined as $f^*(x^*) := \sup_{x} \langle x^*, x \rangle - f(x)$.
A convex, proper, lower-semicontinuous function $f$ is called \emph{co-finite} if $\epi(f)$ contains no non-vertical half-lines, which is equivalent to that $f^*$ is finite everywhere, i.e., that $\dom(f^*) = \RR^n$, \cite[Cor.~13.3.1]{rockafellar1970convex}.
The \emph{subdifferential} of a function $f$ in a point $x$ is the set $\partial f(x) := \{ u \in \RR^n \mid \langle y - x , u \rangle + f(x) \leq f(y)\; \forall \, y \in \RR^n \}$, and if $f$ is proper, convex, and differentiable in $x$ with gradient $\nabla f(x)$, then $\partial f(x) = \{ \nabla f(x) \}$ \cite[Prop.~17.31]{bauschke2017convex}.
A convex, proper, lower-semicontinuous function $f$ is called \emph{essentially smooth} if i) it is differentiable on $\interior(\dom(f))$, i.e., on the interior of the effective domain, and ii) $\lim_{\ell \to \infty} \| \nabla f(x_\ell) \| \to \infty$ for any sequence $\{ x_\ell \}_\ell \subset \interior(\dom(f))$ that either converges to the boundary of $\interior(\dom(f))$ or is such that $\| x_\ell \| \to \infty$.
An operator $A : \RR^n \to \RR^n$ is called strongly monotone if there exists a $\gamma > 0$ such that $\langle Ax - Ay, x - y \rangle \geq \gamma \| x - y \|^2$ for all $x,y\in \RR^n$.
Finally, let  $\{ x_\ell \}_\ell \subset \RR^n$ be a sequence converging to some $\bar{x} \in \RR^n$. The sequence is said to converge Q-linearly if there exists a $\gamma \in (0,1)$ such that $\| x_{\ell + 1} - \bar{x} \| \leq \gamma \| x_{\ell} - \bar{x} \|$, and the sequence is said to converge R-linearly%
\footnote{This is a slightly weaker notion of convergence, compared to Q-linear, that ``is concerned with the overall rate of decrease in the error, rather than the decrease over each individual step of the algorithm'' \cite[pp.~619-620]{nocedal2006numerical}.}
if there exists a sequence of nonnegative numbers $\{ \gamma_\ell \}_\ell \subset \RR_+$ converging Q-linearly to zero and such that  $\| x_\ell - \bar{x} \| \leq \gamma_\ell$ for all $\ell$ \cite[Sec.~9.2]{ortega1970iterative},\cite[pp.~619-620]{nocedal2006numerical}.

\section{Convex graph-structured tensor optimization}\label{sec:convex_graph_tensor_opt}

In this work, we consider a family of optimization problems that generalizes problems of the form \eqref{eq:omt_multi_graph}.
To this end, let $\ccG=(\ccV, \ccE)$ be a connected graph with $\ccT=|\ccV|$ nodes, and let $\bC \in \RRext^{N^\ccT}$ be a cost tensor that takes the form \eqref{eq:structured_cost}. The convex graph-structured tensor optimization problems of interest here are problems of the form
\begin{align}
\minwrt[\bM \in \RR_+^{N^{\ccT}}] & \quad \langle \bC, \bM \rangle + \epsilon D(\bM) + \sum_{t\in \ccV}g_{t}(P_t(\bM)) + \sum_{(t_1,t_2)\in \ccE}f_{t_1,t_2}(P_{t_1,t_2}(\bM)), \label{eq:omt_multi_graph_convex}
\end{align}
where $g_{t}$ and $f_{t_1,t_2}$ are proper, convex, and lower-semicontinuous functionals; further assumptions on these functionals will be imposed where needed.
The reason for our interest in problems of the form \eqref{eq:omt_multi_graph_convex} is that a number of different applications can be modeled as such problems. In particular, this is true for convex dynamic network flow problems (cf.~\cite{haasler2021scalable}), and potential multi-species mean field games. The latter is studied in detail in Section \ref{sec:mean_field_games}.

\begin{remark}
To see that problems of the form \eqref{eq:omt_multi_graph_convex} is a generalization of the graph-structured multi-marginal optimal transport problem \eqref{eq:omt_multi_graph}, let
$\indFun_A(\cdot)$ denote the indicator function on the set $A \subseteq \RR^n$, i.e., the function
\[
\indFun_{A}(x) := \begin{cases}
0, & \text{if } x \in A \\
\infty, & \text{else,}
\end{cases}
\]
and note that this function is proper, convex, and lower-semicontinuous if and only if $A$ is a nonempty, closed, convex set. Now, \eqref{eq:omt_multi_graph} is recovered from \eqref{eq:omt_multi_graph_convex} by taking $g_{t}(P_t(\bM)) = \indFun_{\{ \mu_t \}}(P_t(\bM))$ for $t \in \tilde{\ccV}$ and $g_{t}(P_t(\bM)) \equiv 0$ otherwise, and  $f_{t_1,t_2}(P_{t_1,t_2}(\bM)) \equiv 0$ for all $(t_1,t_2)\in \ccE$.
Other particular cases of interest are unbalanced versions of \eqref{eq:omt_multi_graph} \cite{beier2023unbalanced} or versions of \eqref{eq:omt_multi_graph} where some of the equality constraints are replaced by inequality constraints, cf.~\cite{haasler2021scalable}.

\end{remark}

\begin{remark}
In problem \eqref{eq:omt_multi_graph_convex} the functions $f_{t_1,t_2}$ and the tensor $\bC$ are defined on the same set of edges $\ccE$. This is done for convenience of notation, and is not a restrictive assumption. To see this, note that it is possible to define certain functions $f_{t_1,t_2}$ to be the zero function, or to take certain matrices $C_{t_1,t_2}$ in the decomposition \eqref{eq:structured_cost} to be the zero-matrix.
\end{remark}

Note that \eqref{eq:omt_multi_graph_convex} is typically a large-scale problem, where the full set of variables may neither be stored nor manipulated directly. Therefore one must utilize the problem structure in order to compute the solution.
In this section, we develop a method for such problems, based on generalized Sinkhorn iterations.
This methodology for handling the problem builds on deriving the Lagrangian dual of an optimization problem that is equivalent to \eqref{eq:omt_multi_graph_convex}, and solving this dual using coordinate ascent.
As we will see, the method exploits the graph structure and the algorithm is efficient when the graph is simple, i.e., the tree-width is low (cf.~\cite{haasler2021scalable, fan2022complexity, haasler2020multi}), and when the functionals $g_{t}$ and $f_{t_1,t_2}$ are in some sense simple.

\subsection{An equivalent problem and existence of solution}
We first introduce and analyze a problem that is equivalent to \eqref{eq:omt_multi_graph_convex}, and give conditions under which the latter has an optimal solution. To this end, introducing the variables $\mu_t$, $t \in \ccV$, and $R_{t_1, t_2}$, $(t_1,t_2)\in \ccE$, problem \eqref{eq:omt_multi_graph_convex} can be rewritten as
\begin{subequations}\label{eq:omt_multi_graph_convex_v2}
\begin{align}
\minwrt[\substack{\bM \in \RR_+^{N^{\ccT}},\;  \mu_t \in \RR^N, \; t \in \ccV \\ R_{t_1, t_2} \in \RR^{N \times N}, \; (t_1,t_2) \in \ccE}]
& \; \langle \bC, \bM \rangle + \epsilon D(\bM) + \sum_{t\in \ccV}g_{t}(\mu_t)  + \sum_{(t_1,t_2)\in \ccE}f_{t_1,t_2}(R_{t_1, t_2}) \label{eq:omt_multi_graph_convex_v2_cost} \\
\text{subject to} \qquad & \;  P_t(\bM) =  \mu_t, \; t \in \ccV \label{eq:omt_multi_graph_convex_v2_const1} \\
& \; P_{t_1,t_2}(\bM) = R_{t_1, t_2}, \; (t_1,t_2)\in \ccE. \label{eq:omt_multi_graph_convex_v2_const2}
\end{align}
\end{subequations}
In order for this to be a well-posed problem, we impose the following assumptions on the functionals involved.

\begin{assumption}\label{ass:primal_feasibility_and_optimality}
Assume that all elements of $\bC$ are strictly larger than $-\infty$, and that $g_{t}$, $t\in \ccV$, and $f_{t_1,t_2}$, $(t_1,t_2)\in \ccE$, are all proper, convex, and lower-semicontinuous. Moreover, assume that there exists a feasible point to \eqref{eq:omt_multi_graph_convex_v2} with finite objective function value, i.e., a nonnegative tensor $\bM$ such that $\langle \bC, \bM \rangle < \infty$, and with marginals and bimarginals as in \eqref{eq:omt_multi_graph_convex_v2_const1}-\eqref{eq:omt_multi_graph_convex_v2_const2}, respectively, such that
\begin{align*}
& g_{t}(\mu_t) < \infty, \quad \text{ for all } t\in \ccV,\\
& f_{t_1,t_2}(R_{t_1, t_2}) < \infty, \quad \text{ for all } (t_1,t_2) \in \ccE.
\end{align*}
\end{assumption}

In fact, this assumption ensures that \eqref{eq:omt_multi_graph_convex_v2} has an optimal solution, as stated in the following lemma.

\begin{lemma}\label{lem:primal_optimality}
If Assumption~\ref{ass:primal_feasibility_and_optimality} holds, then there exists a unique optimal solution to problem \eqref{eq:omt_multi_graph_convex_v2}.
\end{lemma}

\begin{proof}
See Appendix~\ref{app:proofs}.
\end{proof}

\begin{remark}
A necessary condition for Assumption~\ref{ass:primal_feasibility_and_optimality} to hold is that there exist vectors $\mu_t \in \RR_+^{N} \cap \dom(g_{t})$, for all $t \in \ccV$, matrices $R_{t_1, t_2} \in \RR_+^{N \times N} \cap \dom(f_{t_1, t_2})$, for all $(t_1,t_2) \in \ccE$, and a constant $\gamma \geq 0$ such that
\begin{align*}
& \mu_t^T \ett = \gamma, \quad \text{ for all } t\in \ccV \\
& R_{t_1, t_2} \ett = \mu_{t_1}, \; R_{t_1, t_2}^T \ett = \mu_{t_2}, \quad \text{ for all } (t_1,t_2) \in \ccE, \\
& \langle C_{t_1,t_2}, R_{t_1, t_2}\rangle < \infty, \quad \text{ for all } (t_1,t_2) \in \ccE.
\end{align*}
However, unless the graph $(\ccV, \ccE)$ is a tree, this is not a sufficient condition for the existence of a tensor that fulfills Assumption~\ref{ass:primal_feasibility_and_optimality}. More precisely, the existence of marginals and bimarginals that are consistent with each other does, in general, not guarantee that there exists a tensor that matches the marginals and bimarginals. A counterexample can be found in \cite[Rem.~3]{haasler2021multimarginal}.
\end{remark}

\subsection{Form of the optimal solution and Lagrangian dual}
Next, we derive the Lagrangian dual of \eqref{eq:omt_multi_graph_convex_v2} and show that there is no duality gap between the primal and the dual problem.

\begin{theorem}\label{thm:lagrangian_dual}
A Lagrangian dual of \eqref{eq:omt_multi_graph_convex_v2} is, up to a constant, given by
\begin{align}
\sup_{\substack{\lambda_t \in \RR^N, \; t \in \ccV \\ \Lambda_{t_1, t_2} \in \RR^{N \times N}, \; (t_1,t_2) \in \ccE}} & \;  -\epsilon \langle \bK, \bU \rangle - \sum_{t \in \ccV}(g_{t})^*(- \lambda_t) - \!\! \sum_{(t_1,t_2)\in \ccE}(f_{t_1,t_2})^*(- \Lambda_{t_1, t_2}), \label{eq:dual}
\end{align}
where $\bK$ and $\bU$ are given by
\begin{subequations}\label{eq:K_and_U}
\begin{align}
\bK^{(i_1 \ldots i_{\ccT})} & = \exp(-\bC^{(i_1 \ldots i_{\ccT})}/\epsilon), \label{eq:K}  \\
\bU^{(i_1 \ldots i_{\ccT})} & = \prod_{t\in \ccV} u_{t}^{(i_t)} \!\!\! \prod_{(t_1,t_2)\in \ccE} \!\!\! U_{t_1, t_2}^{(i_{t_1}, i_{t_2})}  = \prod_{t\in \ccV} \exp \left(\lambda_t^{(i_t)}/\epsilon\right) \!\!\! \prod_{(t_1,t_2)\in \ccE} \!\!\! \exp \left( \Lambda_{t_1, t_2}^{(i_{t_1}, i_{t_2})}/\epsilon \right).  \label{eq:U}
\end{align}
\end{subequations}
Moreover, under Assumption~\ref{ass:primal_feasibility_and_optimality}, the minimum in \eqref{eq:omt_multi_graph_convex_v2} equals the supremum in \eqref{eq:dual} (up to the discarded constant).
Finally, if the dual \eqref{eq:dual} has an optimal solution, then the optimal solution to the primal problem takes the form $\bM^\star = \bK \odot \bU^\star$, where $\bU^\star$ is obtained via \eqref{eq:U} from an optimal solution to \eqref{eq:dual}.
\end{theorem}

\begin{proof}
Relaxing each constraint \eqref{eq:omt_multi_graph_convex_v2_const1} and  \eqref{eq:omt_multi_graph_convex_v2_const2} with a multiplier $\lambda_t \in \RR^N$ and $\Lambda_{t_1, t_2} \in \RR^{N \times N}$, respectively, we get the Lagrangian
\begin{align}
& L(\bM, \mu, R, \lambda, \Lambda) := \langle \bC, \bM \rangle + \epsilon D(\bM) + \sum_{t\in \ccV}g_{t}(\mu_t) + \sum_{(t_1,t_2)\in \ccE}f_{t_1,t_2}(R_{t_1, t_2})  \nonumber \\
& \quad + \sum_{t\in \ccV} \lambda_t^T(\mu_t - P_t(\bM))  + \sum_{(t_1,t_2)\in \ccE} \trace [\Lambda_{t_1, t_2}^T (R_{t_1, t_2} - P_{t_1,t_2}(\bM))], \label{eq:lagrangian}
\end{align}
where $\mu$ denote $(\mu_t)_{t \in \ccV}$, and similar for all other variables.
The dual function is given by $\inf L$ over $\bM$, $\mu$, and $R$, but the Lagrangian decouples over $\bM$, $\mu_t$, and $R_{t_1, t_2}$. For the inf over $\mu_t$ we have that
 \[
 \inf_{\mu_t}  \lambda_t^T\mu_t + g_{t}\!(\mu_t) \! = \! - \! \sup_{\mu_t} (-\!\lambda_t)^T \!\mu_t - g_{t}\!(\mu_t) \! = \! -(g_{t})^*(-\!\lambda_t)
 \]
where $\mbox{}^*$ denotes the Fenchel conjugate; and analogous result follows for $f_{t_1, t_2}$ and the inf over $R_{t_1, t_2}$. This means that
\begin{align}
& \inf_{\bM \geq 0, \mu, R} L(\bM, \mu, R, \lambda, \Lambda)  \nonumber \\
& =  \inf_{\bM \geq 0} \mathcal{L}(\bM, \lambda, \Lambda) - \sum_{t\in \ccV}(g_{t})^*(- \lambda_t) -\sum_{(t_1,t_2)\in \ccE}(f_{t_1,t_2})^*(- \Lambda_{t_1, t_2})  \label{eq:lagrangian_derivation}
\end{align}
where
$
\mathcal{L}(\bM, \lambda, \Lambda) \! := \! \langle \bC, \bM \rangle + \epsilon D(\bM) - \sum_{t\in \ccV} \lambda_t^T P_t(\bM) - \sum_{(t_1,t_2)\in \ccE} \trace [\Lambda_{t_1, t_2}^T P_{t_1,t_2}(\bM)].
$
Noticing that $\lambda_t^T P_t(\bM) =  \sum_{i_t = 1}^N \lambda_t^{(i_t)} \sum_{i_1, \ldots, i_{\ccT} \setminus \{ i_t \}} \bM^{(i_1 \ldots i_{\ccT})} = \sum_{i_1, \ldots, i_{\ccT}} \lambda_t^{(i_t)} \bM^{(i_1 \ldots i_{\ccT})}$, and that $\trace [\Lambda_{t_1, t_2}^T P_{t_1,t_2}(\bM)] = \sum_{i_1, \ldots, i_{\ccT}} \Lambda_{t_1, t_2}^{(i_{t_1}, i_{t_2})} \bM^{(i_1 \ldots i_{\ccT})}$,
we see that $\mathcal{L}(\bM, \lambda, \Lambda)$ decouples over the elements of the tensor. Therefore,
the inf in each element is either attained in $0$, or found by setting the first variation to $0$. If $\bC^{(i_1 \ldots i_{\ccT})} = \infty$, then the trivial case $\bM^{(i_1 \ldots i_{\ccT})} = 0$ holds. Otherwise, setting the first variation equal to $0$ gives
\begin{align*}
0 
= \bC^{(i_1 \ldots i_{\ccT})} + \epsilon \log(\bM^{(i_1 \ldots i_{\ccT})}) - \sum_{t\in \ccV} \lambda_t^{(i_t)} - \!\! \sum_{(t_1,t_2)\in \ccE} \Lambda_{t_1, t_2}^{(i_{t_1}, i_{t_2})}
\end{align*}
from which it follows that then
$\bM^{(i_1 \ldots i_{\ccT})} > 0$.
Moreover, solving for $\bM^{(i_1 \ldots i_{\ccT})}$ gives that $\bM = \bK \odot \bU$, where $\bK$ and $\bU$ are given in \eqref{eq:K_and_U}. Note that this form for $\bM$ also holds for the elements of $\bC$ that are infinite.
Plugging this back into $ \mathcal{L}(\bM, \lambda, \Lambda)$ we get that $\inf_{\bM \geq 0} \mathcal{L}(\bM, \lambda, \Lambda) = -\epsilon \langle \bK, \bU \rangle + N^\ccT \epsilon$, which, after removing the constant, together with \eqref{eq:lagrangian_derivation} gives the dual problem \eqref{eq:dual}.
Finally, for improved readability the detailed proof of that there is no duality gap is deferred to Lemma~\ref{lem:strong_duality} in Appendix~\ref{app:proofs}.
\end{proof}

By using the change of variables implicit in \eqref{eq:U}, problem \eqref{eq:dual} can be expressed equivalently as
\begin{align}
\supwrt[\substack{u_t \in \RR^N_+, \; t \in \ccV \\ U_{t_1, t_2} \in \RR^{N \times N}_+, \; (t_1,t_2) \in \ccE}] & -\epsilon \langle \bK, \bU \rangle - \sum_{t\in \ccV}(g_{t})^*\big(- \epsilon\log(u_t) \big) \nonumber \\
& - \sum_{(t_1,t_2)\in \ccE}(f_{t_1,t_2})^*\big(- \epsilon\log(U_{t_1, t_2}) \big). \label{eq:dual_U}
\end{align}
Moreover, under a Slater-type condition of for the primal problem, i.e., that the relative interior (denoted $\ri$)%
\footnote{The relative interior of a set $A$ consists of all points in $A$ that are interior when $A$ is regarded as a subset of its affine hull, see \cite[Ch.~6]{rockafellar1970convex}.}
of the effective domains of the cost functions in \eqref{eq:omt_multi_graph_convex} have a nonempty intersection,
we have that the suprema in \eqref{eq:dual} and \eqref{eq:dual_U} are attained.

\begin{assumption}\label{ass:slater}
Assume that there exists an $\bM > 0$ such that $\langle \bC, \bM \rangle < \infty$, and with marginals $(\mu_t)_{t \in \ccV}$ and bimarginals $(R_{t_1, t_2})_{(t_1, t_2) \in \ccE}$ satisfying \eqref{eq:omt_multi_graph_convex_v2_const1} and \eqref{eq:omt_multi_graph_convex_v2_const2}, respectively, so that
\begin{itemize}
\item for all $g_{t}$ and $f_{t_1, t_2}$ that are polyhedral, $\mu_t \in \dom(g_{t})$ and $R_{t_1, t_2} \in \dom(f_{t_1, t_2})$, 
\item for all $g_{t}$ and $f_{t_1, t_2}$ that are not polyhedral, $\mu_t \in \ri(\dom(g_{t}))$ and $R_{t_1, t_2} \in \ri(\dom(f_{t_1, t_2}))$.
\end{itemize}
\end{assumption}

\begin{corollary}\label{cor:slater}
Given Assumption~\ref{ass:slater} the conclusions of Theorem~\ref{thm:lagrangian_dual} hold, with the addition that the dual \eqref{eq:dual} is guaranteed to have a nonempty set of optimal solutions.
\end{corollary}

\begin{proof}
The result follows from \cite[Ch.~29 and 30]{rockafellar1970convex}.
\end{proof}

Even if the Slater-type condition in Assumption~\ref{ass:slater} is not fulfilled, the form $\bM = \bK \odot \bU$ will be important in deriving a convergent algorithm for solving \eqref{eq:omt_multi_graph_convex_v2}.

\subsection{Coordinate ascent iterations for solving the dual problem}

In this section we derive an efficient solution method for \eqref{eq:omt_multi_graph_convex_v2}, based on performing coordinate ascent in the dual problem \eqref{eq:dual} (or, equivalently, in \eqref{eq:dual_U}).
To this end, let $\phi( (\lambda_{t})_{t \in \ccV}, (\Lambda_{t_1, t_2})_{(t_1, t_2) \in \ccE})$ denote the objective function in the dual problem \eqref{eq:dual}.
Given an iterate $((\lambda_{t}^{k})_{t \in \ccV}, (\Lambda_{t_1, t_2}^k)_{(t_1, t_2) \in \ccE})$, in a coordinate ascent step we cyclically select an element $j \in \ccV$ or $(j_1, j_2) \in \ccE$ and compute an update to the corresponding variable by taking $\lambda_{j}^{k+1}$ to be in
\begin{subequations}\label{eq:argmax_lambda}
\begin{equation}\label{eq:argmax_small_lambda}
\argmax_{\lambda_j \in \RR^N} \quad \phi( \lambda_{j}, (\lambda_{t}^k)_{t \in \ccV \setminus \{ j \}}, (\Lambda_{t_1, t_2}^k)_{(t_1, t_2) \in \ccE}) ,
\end{equation}
or $\Lambda_{j_1, j_2}^{k+1}$ to be in
\begin{equation}\label{eq:argmax_capital_lambda}
\argmax_{\Lambda_{j_1, j_2} \in \RR^{N \times N}} \phi( \Lambda_{j_1, j_2}, (\lambda_{t}^k)_{t \in \ccV}, (\Lambda_{t_1, t_2}^k)_{(t_1, t_2) \in \ccE \setminus \{(j_1, j_2)\}}),
\end{equation}
\end{subequations}
respectively, while taking  $\lambda_{t}^{k+1} = \lambda_{t}^{k}$ and $\Lambda_{t_1, t_2}^{k+1} = \Lambda_{t_1, t_2}^k$ for all other elements. In order for this to be a well-defined algorithm, we need that the set of maximizing arguments in \eqref{eq:argmax_lambda} is always nonempty.
To guarantee this, we impose the following assumption (which is milder than Assumption~\ref{ass:slater}).

\begin{assumption}\label{ass:dual_ascent_attainment}
Assume that $\bC < \infty$ and that for each index $j \in \ccV$ there exists a $\mu_j > 0$ so that
\begin{itemize}
\item if $g_{j}$ is polyhedral, then $\mu_j \in \dom(g_{j})$, 
\item if $g_{j}$ is not polyhedral, then $\mu_j \in \ri(\dom(g_{j}))$,
\end{itemize}
and analogously for each index $(j_1, j_2) \in \ccE$, $R_{j_1, j_2}$, and $f_{j_1, j_2}$.
\end{assumption}

\begin{lemma}\label{lem:argmax_lambda}
Under Assumptions~\ref{ass:primal_feasibility_and_optimality} and \ref{ass:dual_ascent_attainment}, the subproblems in \eqref{eq:argmax_lambda} always have a nonempty set of maximizers.
\end{lemma}

\begin{proof}
To prove the lemma, we restrict our attention to one subproblem of the form \eqref{eq:argmax_small_lambda}; for subproblems of the form \eqref{eq:argmax_capital_lambda} it follows analogously. Now, note that problem \eqref{eq:argmax_small_lambda} can be see as the Lagrangian dual of the primal problem
\begin{align*}
\minwrt[\bM \in \RR_+^{N^{\ccT}},\;  \mu_{j} \in \RR^N]
& \; \langle \bC, \bM \rangle + \epsilon D(\bM) + g_{j}(\mu_{j}) - \!\!\! \sum_{t \in \ccV \setminus \{ j \}} \!\! (\lambda_t^k)^T P_t(\bM) \\
& \quad - \sum_{(t_1,t_2)\in \ccE} \trace [(\Lambda_{t_1, t_2}^{k})^T P_{t_1,t_2}(\bM)] \\
\text{subject to} \quad & \; P_{j}(\bM) =  \mu_{j}.
\end{align*}
Moreover, using Assumption~\ref{ass:dual_ascent_attainment} we have that $\mu_j > 0$ and $\bM = \mu_j \otimes (\otimes_{t \in \ccV \setminus \{ j \}} \ett) > 0$ is a point fulfilling Slater's condition for the above problem. Therefore, following \cite[Ch.~29 and 30]{rockafellar1970convex}
we have that strong duality holds between these two problems, and in particular that the dual \eqref{eq:argmax_small_lambda} has a nonempty set of maximizers (cf.~\cite[Lem.~3.1]{tseng1993dual}).
\end{proof}

By the above lemma, the coordinate ascent steps in \eqref{eq:argmax_lambda} are well-defined. Moreover, since each problem is concave and unconstrained, the optimal solution is where the subgradient is zero.
To compute the subgradients, first note that
\[
P_j(\bK \odot \bU) \oslash u_j = \!\!
\sum_{i_1, \ldots, i_{\ccT} \setminus i_{j}} \!\! \bK^{(i_1 \ldots i_\ccT)} \!\!
\prod_{t\in \ccV \setminus \{j\} } \!\! u_{t}^{(i_t)} \!\!
\prod_{(t_1,t_2)\in \ccE} \!\! U_{t_1, t_2}^{(i_{t_1}, i_{t_2})}
\]
is a well-defined vector which is independent of $u_j$. We therefore define
\begin{subequations}\label{eq:script_P}
\begin{equation}\label{eq:script_P_marginal}
w_j := P_j(\bK \odot \bU) \oslash u_j,
\end{equation}
and note that this means that $P_j(\bK \odot \bU) = u_j \odot w_j$.
Analogously, we also define
\begin{equation}\label{eq:script_P_bimarginal}
W_{j_1, j_2} := P_{j_1, j_2}(\bK \odot \bU) \oslash U_{j_1, j_2},
\end{equation}
\end{subequations}
which in the same way is a well-defined matrix, independent of $U_{j_1, j_2}$, and hence $P_{j_1, j_2}(\bK \odot \bU) = U_{j_1, j_2} \odot W_{j_1, j_2}$.

Next, note that 
$
\partial \langle \bK, \bU \rangle/\partial \lambda_{j}^{(i_{j})}
= - \exp\left( \lambda_{j}^{(i_{j})} /\epsilon \right) w_{j}^{(i_{j})} = - u_{j}^{(i_{j})} w_{j}^{(i_{j})}
$
with $\bK$ and $\bU$ given as in \eqref{eq:K_and_U} and $w_{j}$ as in \eqref{eq:script_P_marginal}. Thus, in each update of the variable $\lambda_{j}$ one has to solve the inclusion problem%
\begin{subequations}\label{eq:sinkhorn_inclusion}
\begin{equation}\label{eq:sinkhorn_inclusion_marginal}
0 \in \partial_{\lambda_{j}} \phi \! = \! - \exp\left( \lambda_{j} /\epsilon \right) \odot w_{j} + \partial (g_{j})^*(- \lambda_{j}),
\end{equation}
where $\partial_{\lambda_{j}}$ denotes the subdifferential with respect to $\lambda_{j}$. By an analogous derivation, in each update of the variable $\Lambda_{j_1, j_2}$ one has to solve the inclusion problem
\begin{align}
0 \in  \partial_{\Lambda_{j_1, j_2}} \phi = &
- \exp\left( \Lambda_{j_1, j_2} /\epsilon \right) \odot W_{j_1, j_2} + \partial(f_{j_1, j_2})^*(- \Lambda_{j_1, j_2}).
\label{eq:sinkhorn_inclusion_bimarginal}
\end{align}
\end{subequations}
To verify that the two equities in \eqref{eq:sinkhorn_inclusion} hold, see, e.g., \cite[Cor.~16.38]{ bauschke2017convex}
These inclusions, and hence the updates, can be reformulated in terms of the transformed dual variables $u_{j}$ and $U_{j_1, j_2}$, in which case they read
\begin{subequations}\label{eq:sinkhorn_inclusion_u}
\begin{align}
& 0 \in - u_{j} \odot w_{j} + \partial (g_{j})^*\big(- \epsilon \log( u_{j})\big), \label{eq:sinkhorn_inclusion_marginal_u} \\
& 0 \in - U_{j_1, j_2} \odot W_{j_1, j_2} + \partial(f_{j_1, j_2})^*\big(-\! \epsilon\log(U_{j_1, j_2}) \big). \label{eq:sinkhorn_inclusion_bimarginal_U}
\end{align}
\end{subequations}
This is summarized in Algorithm~\ref{alg:generalized_sinkhorn}.
However, note that directly computing $w_j$ and $W_{j_1, j_2}$ needed in \eqref{eq:sinkhorn_inclusion_u} by brute-force is computationally demanding, and effectively numerically infeasible for large-scale problems. Therefore, from this perspective Algorithm~\ref{alg:generalized_sinkhorn} is an ``abstract algorithm''. Nevertheless, for many graph structures it is possible to compute the projections efficiently
by sequentially eliminating the modes of the tensor, see \cite{benamou2015bregman, haasler2020optimal, elvander2020multi, haasler2021multimarginal, haasler2021scalable, haasler2020multi,haasler2021control, singh2020inference, altschuler2020polynomial}.
In particular, in Section \ref{sec:mean_field_games} we show how this is done for the application of multi-species potential mean field games (see Algorithm~\ref{alg:multi_species}). Moreover,
storing and using intermediate results of eliminated modes, the procedure can also be understood as a message-passing scheme \cite{haasler2020multi}.
Finally, under relatively mild assumptions, Algorithm~\ref{alg:generalized_sinkhorn} is convergent in the following sense.

\begin{algorithm}[tb]
  \begin{algorithmic}[1]
    \STATE Give: graph $\ccG=(\ccV,\ccE)$, cost tensor $\bC$ that decouples according to $\ccG$, functions $(g_{t})^*$, for $t \in \ccV$, and $(f_{t_1, t_2})^*$, for $(t_1, t_2) \in \ccE$, nonnegative initial guesses $(u_{t}^{0})_{t\in \ccV}$ and $(U_{t_1, t_2}^{0})_{(t_1,t_2)\in \ccE}$.
    \STATE $k = 0$
    \WHILE{Not converged}	
    \STATE $k = k+1$
    \FOR{$j \in \ccV$ and $(j_1, j_2)\in \ccE$}
    \STATE Update $u_j^k$ by solving \eqref{eq:sinkhorn_inclusion_marginal_u} with $w_{j}$ as in \eqref{eq:script_P_marginal}.
    \STATE Update $U_{j_1, j_2}^k$ by solving \eqref{eq:sinkhorn_inclusion_bimarginal_U} with $W_{j_1, j_2}$ as in \eqref{eq:script_P_bimarginal}.
    \ENDFOR    
    \ENDWHILE
    \RETURN $(u_{t}^k)_{t\in \ccV}$ and $(U_{t_1, t_2}^k)_{(t_1,t_2)\in \ccE}$.
  \end{algorithmic}
  \caption{Generalized Sinkhorn method for solving \eqref{eq:omt_multi_graph_convex_v2}.}
  \label{alg:generalized_sinkhorn} 
\end{algorithm}

\begin{theorem}\label{thm:convergence}
Given Assumptions~\ref{ass:primal_feasibility_and_optimality} and \ref{ass:dual_ascent_attainment}, and assume further that
\begin{enumerate}
\item  $g_{t}$, $t \in \ccV$, and $f_{t_1, t_2}$, $(t_1,t_2) \in \ccE$, are all continuous on $\dom(g_{t})$ and $\dom(f_{t_1, t_2})$, respectively,
\item for all $g_{t}$, $t \in \ccV$, and $f_{t_1, t_2}$, $(t_1,t_2) \in \ccE$, that are not polyhedral, the feasible point in Assumption~\ref{ass:primal_feasibility_and_optimality} is such that $\mu_t \in \ri(\dom(g_{t}))$ and $R_{t_1, t_2} \in \ri(\dom(f_{t_1, t_2}))$, respectively. 
\end{enumerate}
 Let $(u_{t}^k)_{t\in \ccV}$ and $(U_{t_1, t_2}^k)_{(t_1,t_2)\in \ccE}$ be the iterates of Algorithm~\ref{alg:generalized_sinkhorn} at iteration $k$, and let $\bU^k$ be the corresponding tensor as in \eqref{eq:U}. Moreover, let $\bM^k = \bK \odot \bU^k$. Then $(\bM^k)_k$ is a bounded sequence that converges to the optimal solution to \eqref{eq:omt_multi_graph_convex_v2}.
Furthermore, if the set of optimal solutions to \eqref{eq:dual} is nonempty and bounded, then $((u_{t}^k)_{t\in \ccV}, (U_{t_1, t_2}^k)_{(t_1,t_2)\in \ccE})_k$ is a bounded sequence and every cluster point is an optimal solution to \eqref{eq:dual_U}.
\end{theorem}

\begin{proof}
To prove the theorem, let  
$h(\bM) := \langle \bC, \bM \rangle + \epsilon D(\bM)$, which is a strictly convex function (cf.~\cite[Ex.~9.35]{bauschke2017convex}). Moreover $\dom(h) = \RR_{+}^{N^\ccT}$, and hence polyhedral.
Next, we observe that $h$ is co-finite, since the Fenchel conjugate of $h$ is given by%
\footnote{Compare with the expression $\inf_{\bM} \mathcal{L}(\bM, \lambda, \Lambda) = -\epsilon \langle \bK, \bU \rangle + N^\ccT \epsilon$ in the proof of Theorem~\ref{thm:lagrangian_dual}.}
\begin{align*}
h^*(\bT) & = -\epsilon \sum_{i_1 \ldots i_\ccT} \exp((\bT^{(i_1 \ldots i_\ccT)} - \bC^{(i_1 \ldots i_\ccT)})/\epsilon) - 1 = - \epsilon \langle \bK, \exp(\bT/\epsilon) \rangle + N^\ccT \epsilon,
\end{align*}
see \cite[Ex.~13.2 and Prop.~13.23]{bauschke2017convex}.
Therefore, following along the lines of \cite[Sec.~6]{tseng1993dual}, we have that $(\bM^k)_k$ is a bounded sequence and that every cluster point is an optimal solution to \eqref{eq:omt_multi_graph_convex_v2}. In particular, \cite[Thm.~3.1]{tseng1993dual} imposes some slightly stronger assumptions,%
\footnote{More precisely, to directly apply the result in \cite[Thm.~3.1]{tseng1993dual}, we must assume that the feasible point in Assumption~\ref{ass:primal_feasibility_and_optimality} is such that $\bM > 0$; see \cite[Ass.~B]{tseng1993dual} where ``$f_0$'' corresponds to $\langle \bC, \bM \rangle + \epsilon D(\bM)$. For an example of where this weaker assumption is indeed used, see Example~\ref{ex:no_positive_M}.}
but it is readily checked in all places where these stronger assumptions are invoked that the same conclusions hold true in this particular case under the weaker assumptions. For brevity, we omit the details of the modifications needed.

Since $(\bM^k)_k$ is a bounded sequence and every cluster point is optimal to \eqref{eq:omt_multi_graph_convex_v2}, by the uniqueness of the optimal solution $\bM^\star$ the sequence must converge to it.
To see this, note that since $(\bM^k)_k$ is bounded, if it does not converge then it must have at least two cluster points. This is a contradiction, since every cluster points must be optimal, and the optimal solution is unique.
Finally, the last statement of the theorem follows similarly from \cite[Thm.~3.1(b)]{tseng1993dual}.
\end{proof}

The above theorem guarantees convergence, but does not guarantee how fast the iterates converge. In particular, in order to guarantee R-linear convergence
we need to impose further assumptions on the functions involved.

\begin{theorem}\label{thm:linear_convergence_v2}
Given Assumption~\ref{ass:primal_feasibility_and_optimality}, further assume that there exists an $\bM > 0$ with marginals and bimarginals $(\mu_t)_{t \in \ccV}$ and $(R_{t_1, t_2})_{(t_1, t_2) \in \ccE}$ satisfying \eqref{eq:omt_multi_graph_convex_v2_const1} and \eqref{eq:omt_multi_graph_convex_v2_const2}, respectively, and that all functions $g_{t}$ and $f_{t_1, t_2}$ are such that either
\begin{enumerate}[i)]
\item the function is a polyhedral indicator function and $\mu_t \in \dom(g_{t})$ or $R_{t_1, t_2} \in \dom(f_{t_1, t_2})$, respectively, or
\item the function is co-finite, essentially smooth, continuous on the effective domain, and the gradient operator is strongly monotone and Lipschitz continuous on any compact convex subset of the interior of the effective domain, and so that $\mu_t \in \interior(\dom(g_{t}))$ or $R_{t_1, t_2} \in \interior(\dom(f_{t_1, t_2}))$, respectively.
\end{enumerate}
Under these assumptions, let $(u_t^k)_{t \in \ccV}$ and $(U_{t_1, t_2}^k)_{(t_1, t_2) \in \ccE}$ be the iterates of Algorithm~\ref{alg:generalized_sinkhorn}, and let $\bM^k = \bK \odot \bU^{k}$.
Then $\bM^k \to \bM^\star$ at least R-linearly, where $\bM^\star$ is the unique optimal solution to \eqref{eq:omt_multi_graph_convex_v2}, and the cost function in \eqref{eq:dual_U}, evaluated in $(u_t^k)_{t \in \ccV}$ and $(U_{t_1, t_2}^k)_{(t_1, t_2) \in \ccE}$, converges to the optimal value of \eqref{eq:omt_multi_graph_convex_v2} at least R-linearly.
\end{theorem}

\begin{proof}
Assume first that all functions are as in ii). In this case, note that \eqref{eq:omt_multi_graph_convex_v2_cost} is separable in the different variables, and that $E$ in \cite[Eq.~(1.1)]{luo1993convergence} is of the form
\[
E^T = \left[ \begin{array}{ccc|ccc}
\mathtt{P}_1^T & \ldots & \mathtt{P}_{\ccT}^T  & \mathtt{P}_{1,2}^T & \ldots & \mathtt{P}_{\ccT, \ccT-1}^T \\
\hline
      & -I     &            &           & 0 \\
\hline
      & 0      &            &           & -I
\end{array}
\right]
\]
where $\mathtt{P}_t$ is a matrix so that $\mathtt{P}_t \text{vec}(\bM)$ is the projection on the $t$th marginal and $\mathtt{P}_{t_1, t_2}$ is a matrix such that $\mathtt{P}_{t_1, t_2} \text{vec}(\bM)$ is the projection on the $(t_1, t_2)$-bimarginal. This means that $\mathtt{P}_t^T$ and $\mathtt{P}_{t_1, t_2}^T$ are the corresponding back-projections.
Now, under the given assumptions the results in \cite[Thm~6.1]{luo1993convergence} are directly applicable.

In the case that some of the functions are of the form as in i), this cost function can be replaced by a finite number of inequality constraints.
By adding the corresponding inequalities in the matrix $E$ above, the above argument show R-linear convergence of the algorithm.
\end{proof}

\begin{remark}
One assumption in Theorem~\ref{thm:linear_convergence_v2} is that all functions $g_{t}$ and $f_{t_1, t_2}$ (that are not polyhedral indicator functions) are such that they are differentiable on the interior of their effective domains.
Under this assumption, all inclusions in \eqref{eq:sinkhorn_inclusion} and \eqref{eq:sinkhorn_inclusion_u} are in fact equalities on the interior of the effective domain.
\end{remark}

\begin{example}\label{ex:no_positive_M}
To illustrate some of the differences between the results presented so far, here we consider a small bimarginal example.
To this end, let $\bM, \bC \in \RR^{2 \times 2}$, and consider the problem
\[
\minwrt[\bM \in \RR_+^{2 \times 2}] \; D(\bM) \; \text{ subject to } \; P_1(\bM) \leq \begin{bmatrix} 1 \\ 2 \end{bmatrix}, \; P_{12}(\bM) \geq
\begin{bmatrix}
1 & 0\\
0 & 0
\end{bmatrix},
\]
where we for simplicity have taken $\bC = 0$ and $\epsilon = 1$. 
The two constraints together imply that $\bM^{(1,2)} = 0$ for any feasible solution, and hence neither the conditions in Assumption~\ref{ass:slater} nor the ones in Theorem~\ref{thm:linear_convergence_v2} are fulfilled.
Nevertheless,
the conditions in Assumption~\ref{ass:primal_feasibility_and_optimality} are fulfilled, and hence the problem has a unique optimal solution (Lemma~\ref{lem:primal_optimality}); the latter is given by
\[
\bM^\star = \begin{bmatrix}
1 & 0 \\ 1 & 1
\end{bmatrix}.
\]
 Moreover, the conditions in Assumption~\ref{ass:dual_ascent_attainment} are fulfilled, and hence each step in the algorithm is therefore well-defined (Lemma~\ref{lem:argmax_lambda}). In fact, the conditions in Theorem~\ref{thm:convergence} are fulfilled, which guarantees that the dual ascent algorithm is converging to
the optimal solution.
For suitable initial conditions the coordinate ascent method gives the iterates
\[
u_1^{k} = \begin{bmatrix}
1/(\exp(k) + 1) \\ 1
\end{bmatrix}, \qquad
U_{1,2}^k = 
\begin{bmatrix}
\exp(k) & 1 \\
1 & 1
\end{bmatrix},
\]
and the corresponding dual cost
converges towards the optimal value as $k \to \infty$. However, the dual problem does not attain an optimal solution since  $(U^{(1,2)})^k$ diverges as $k \to \infty$.
Finally, by evaluating $\| \bM^{k} - \bM^\star \|_2$ it can be seen that in fact the iterates converge R-linearly, which indicates that there might be room for improvement with respect to the conditions in Theorem~\ref{thm:linear_convergence_v2}.
\end{example}

As a final remark, note that Assumptions~\ref{ass:slater} and \ref{ass:dual_ascent_attainment} both enforce that we must have $\bC < \infty$; the first one implicitly and the second one explicitly. Similarly, the functions $g_{t}$ and $f_{t_1, t_2}$ must have effective domains that include marginals and bimarginals that are elementwise strictly positive, and hence they cannot, e.g., be indicator functions on singletons with zero elements. For some applications this is not fulfilled, and in particularly this is the case for the example in Section~\ref{sec:mean_field_games}.
Nevertheless, the assumptions can be weakened somewhat to accommodate for this, similar to \cite[Sec.~4.1]{haasler2021scalable}.
More specifically, if any element $\bC^{(i_1 \ldots i_\ccT)} = \infty$, then we can fix $\bM^{(i_1 \ldots i_\ccT)} = 0$ and remove it from the set of variables. This means that $\bM$ is technically no longer a tensor, but the marginal and bimarginal projections can still be defined, and the above derivations carry over to this setting. Similarly, if $\dom(g_{j})$ is such that $\mu_j^{(i_j)} = 0$, then we can remove all the variables $\bM^{(i_1 \ldots i_\ccT)}$ with indices $\{ (i_1, \ldots, i_{j-1}, i_{j}, i_{j+1}, \ldots, i_\ccT) \mid i_t = 1, \ldots, N \text{ for } t \neq j \}$, and analogously for $f_{j_1, j_2}$ and the bimarginals.
From the perspective of Algorithm~\ref{alg:generalized_sinkhorn}, it is interesting to note that in the first case $\bK^{(i_1 \ldots i_\ccT)}=0$, and in the second case we can take $u_j^{(i_j)} = 0$.

\subsection{Extension to multiple costs on each marginal}\label{sec:multiple_costs}

In some problems, marginals and bimarginals can be associated with multiple functions, typically when they are both associated with a cost and an inequality constraint. 
To handle such cases, we consider a modified version of problem \eqref{eq:omt_multi_graph_convex_v2} that takes the form%
\footnote{For ease of notation, we have the same number of functions $\kappa_1$ and $\kappa_2$ associated with each marginal and bimarginal, respectively, however this can easily be relaxed. Moreover, note that the constraints implicitly ensure that $\mu_{t, k_1} = \mu_{t, k_1'}$, for all $k_1, k_1' = 1, \ldots, \kappa_1$ and all $t \in \ccV$, for any feasible point, and similarly for the bimarginals.}
\begin{align}
\minwrt[\substack{
\bM \in \RR_+^{N^{\ccT}}, \; \mu_{t, k_1} \in \RR_+^N, \\
R_{t_1, t_2, k_2} \in \RR_+^{N \times N} \\
t \in \ccV \text{ and } k_1 = 1, \ldots, \kappa_1 \\
(t_1,t_2) \in \ccE \text{ and } k_2 = 1, \ldots, \kappa_2
}]
& \langle \bC, \bM \rangle + \epsilon D(\bM) + \sum_{t\in \ccV} \sum_{k_1=1}^{\kappa_1} g_{t, k_1}(\mu_{t, k_1}) \nonumber \\[-25pt]
& \quad + \!\!\sum_{(t_1,t_2)\in \ccE} \sum_{k_2=1}^{\kappa_2} f_{t_1,t_2, k_2}(R_{t_1, t_2, k_2}) \label{eq:omt_multi_graph_convex_multiple_costs_v2}\\
\text{subject to} \qquad & P_t(\bM) =  \mu_{t, k_1}, \; k_1 = 1,\ldots,\kappa_1, \; t \in \ccV \nonumber \\
& P_{t_1,t_2}(\bM) = R_{t_1, t_2, k_2}, \; k_2 = 1, \ldots, \kappa_2, \; (t_1,t_2)\in \ccE. \nonumber
\end{align}
By modifying the arguments in the previous sections, it is straightforward to derive a Lagrangian dual of \eqref{eq:omt_multi_graph_convex_multiple_costs_v2} and to see that if the dual problem has an optimal solution, then the optimal solution to \eqref{eq:omt_multi_graph_convex_multiple_costs_v2} is of the form $\bM = \bK \odot \bU$, where
\begin{align*}
& \bU^{(i_1 \ldots i_{\ccT})} = \left(\prod_{t\in \ccV} \prod_{k_1=1}^{\kappa_1} u_{t, k_1}^{(i_t)}\right)\left( \prod_{(t_1,t_2)\in \ccE} \prod_{k_2 = 1}^{\kappa_2} U_{t_1, t_2, k_2}^{(i_{t_1}, i_{t_2})}\right)
\end{align*}
and where $\bK$ is as before (see \eqref{eq:K}). This can be interpreted as splitting $u_{t}$ as $u_{t} = \odot_{k_1=1}^{\kappa_1} u_{t, k_1}$, and similarly $U_{t_1, t_2} = \odot_{k_2=1}^{\kappa_2} U_{t_1, t_2, k_2}$. Moreover, this means that the coordinate ascent inclusion for $u_{j, \tilde{k}_1}$ is given by
\begin{align*}
& 0 \in - u_{j, \tilde{k}_1} \odot \! \Big(\bigodot_{k_1 \neq \tilde{k}_1}  u_{j, k_1} \!  \Big) \! \odot w_{j} + \partial (g_{j, \tilde{k}_1})^*\big(- \epsilon \log( u_{j, \tilde{k}_1})\big),
\end{align*}
where $w_{j}$ is defined analogously to \eqref{eq:script_P_marginal} as
\begin{align}
w_{j} & =  P_{j}(\bK \odot \bU) \oslash \Big(\bigodot_{k_1 = 1}^{\kappa_1}  u_{j, k_1}  \Big). \label{eq:script_P_modified}
\end{align}
Similar expressions hold for the inclusion problem for $U_{t_1, t_2, \tilde{k}_2}$.
Furthermore, reexamining the proof of Theorem~\ref{thm:convergence} and \ref{thm:linear_convergence_v2}, it can be readily seen that by modifying the assumptions accordingly, the results can be extended to this setting. For brevity, we omit explicitly stating these results.
Finally, by reexamining the argument of sequentially eliminating the modes of the tensor as in \cite{elvander2020multi, haasler2021multimarginal, haasler2021scalable}, one can see that the efficiency in computing $w_j$ in \eqref{eq:script_P_modified} (and also $W_{j_1, j_2}$) only depends on the underlying graph structure $(\ccV, \ccE)$, and not on the number of cost functions associated with each marginal (and bimarginal). Therefore, we can still efficiently solve the inclusions for ``simple functions'' and graph structures for which the projections can be easily computed.

\section{Multi-species potential mean field games}\label{sec:mean_field_games}
An important tool for analyzing and controlling systems of systems, which has emerged during the last decades, is mean field games \cite{jovanovic1988anonymous, huang2006large, lasry2007mean, huang2012social, caines2018mean, djehiche2017mean}.
Mean field games are models of dynamic games where each player's action is negligible to other players at the individual level, but where the actions are significant when aggregated.
A subclass of such games are potential mean field games. These can be seen as density control problems, where the density abides to a controlled Fokker-Planck equation with distributed control \cite{lasry2007mean}. This type of control problems have been studied in, e.g., \cite{cardaliaguet2015second, bensoussan2013mean, chen2018steering}.
An important generalization of mean field games is the multi-species setting, where the population consists of several different types of agents or species
\cite{huang2006large, lasry2007mean, lachapelle2011mean, achdou2017mean, cirant2015multi, bensoussan2018mean}.
In this section, we show that discretizations of potential multi-species mean field games take the form of a convex graph-structured tensor optimization problem \eqref{eq:omt_multi_graph_convex_v2}. 
By also deriving efficient methods for computing the corresponding  projections needed in Algorithm~\ref{alg:generalized_sinkhorn}, we here develop an efficient numerical solution algorithm for solving such problems.
In order to do so, we will first consider the nonlinear density control problem obtained in the single-species setting, and its corresponding discretization.

\subsection{The single-species problem}
Let $X \subset \RR^n$ be a state space, and consider a set of infinitesimal agents on $X$ 
which obeys the (It\^o) stochastic differential equation
\begin{equation}\label{eq:SDE}
dx(t) = f(x(t))dt + B(x(t))\big(v(x(t), t)dt + \sqrt{\epsilon}dw\big),
\end{equation}
with initial condition $x(0) = x_0 \sim \rho_0(x)$, where $w$ is a $m$-dimensional Wiener process.
More precisely,
assume that $f : X \to \RR^n$ and $B : X \to \RR^{n \times m}$ are continuously differentiable with bounded derivatives,
in which case,
under suitable conditions on the (Markovian) feedback $v$, there exists a unique solution to \eqref{eq:SDE} almost surely, see, e.g, \cite[Thm.~V.4.1]{fleming1975deterministic}, \cite[pp.~7-8]{bensoussan2013mean}.
Moreover, under suitable regularity conditions \cite{bensoussan2013mean, cardaliaguet2015second} the density $\rho(t, \cdot)$ which describe the distribution of particles at time point $t$ exists and is the solution of a controlled Fokker-Planck equation (cf.~\cite[p.~72]{astrom1970introduction}).
A potential mean field game can then be reformulated as the density optimal control problem \cite{lasry2007mean}
\begin{subequations}\label{eq:densitycontrol}
    \begin{align}
        \minwrt[\rho, v] & \int_0^1 \!\!\! \int_X \! \frac{1}{2} \|v\|^2 \! \rho dx dt + \!\! \int_0^1 \!\!\! \mathcal{F}_t(\rho(t,\cdot)) dt + \mathcal{G}(\rho(1,\cdot)) \label{eq:densitycontrol1} \\
       \text{subject to } & \, \frac{\partial \rho}{\partial t} + \nabla\cdot((f + Bv) \rho )- \frac{\epsilon}{2} \sum_{i,k=1}^n \frac{\partial^2 (\sigma_{ik}  \rho)}{\partial x_i \partial x_k} = 0 \label{eq:densitycontrol2} \\
        & \, \rho(0, \cdot) = \rho_0. \label{eq:densitycontrol3}	
    \end{align}
\end{subequations}
Here, $\sigma(x) := B(x)B(x)^T$. Moreover, $\mathcal{F}_t$ and $\mathcal{G}$ are functionals on $L_2\cap L_\infty$, and we assume that they are proper, convex, and lower-semicontinuous. We also assume that $\mathcal{F}_t$ is piece-wise continuous with respect to $t$.

To discretize problem \eqref{eq:densitycontrol}, we rewrite it as a problem over path space. To this end,
let $\ccP^v$ denote the distribution on path space, i.e., a probability distribution over $C([0,1], X)$ := the set of continuous functions from $[0,1]$ to $X$, induced by the controlled process \eqref{eq:SDE}. In particular, this means that for the marginal of $\ccP^v$ corresponding to time $t$, denoted $\ccP^v_t$, we have that $\ccP^v_t = \rho(t, \cdot)$, where $\rho$ is the solution to \eqref{eq:densitycontrol2} and \eqref{eq:densitycontrol3}.
Moreover, let $\ccP^0$ denote the corresponding (uncontrolled) Wiener process with initial density $\rho_0$. 
By the Girsanov theorem (see, e.g., \cite[pp.~156-157]{Follmer88}, \cite[p.~321]{dai1991stochastic}), we have that
	\begin{equation}\label{eq:relativeentropy}
		\frac{1}{2}\!\int_X\! \int_0^1 \! \|v\|^2 \rho dt dx
		\! = \! \frac{1}{2} \RE_{\ccP^v} \left\{ \int_0^1 \! \|v\|^2 dt \right\} \! = \! 
		\epsilon {\rm KL} (\ccP^v \| \ccP^0) 
	\end{equation}
where ${\rm KL} (\cdot \| \cdot)$ is the Kullback-Leibler divergence, see, e.g., \cite{chen2016relation, gentil2017analogy, chen2018steering, BenCarDiNen19,leonard2012schrodinger, leonard2013schrodinger}.
To ensure that \eqref{eq:relativeentropy} holds, it is important that the control signal and the noise enter the system through the same channel,  as in \eqref{eq:SDE} \cite{chen2016optimalPartI, chen2017optimal}. Moreover, the link between stochastic control and entropy provided by \eqref{eq:relativeentropy} has recently led to several novel applications of optimal control \cite{chen2016relation, chen2016optimalPartI, chen2017optimal, chen2020stochastic, caluya2021wasserstein}.

By using \eqref{eq:relativeentropy}, the problem \eqref{eq:densitycontrol} can be reformulated as 
\begin{subequations}\label{eq:densitycontrol_KL_formulation}
    \begin{align}
    \minwrt[\ccP^v] & \quad \epsilon\, {\rm KL} (\ccP^v \| \ccP^0) + \int_0^1 \mathcal{F}_t(\ccP^v_t)dt + \mathcal{G}(\ccP^v_1) \\
    \text{subject to } & \quad  \ccP^v_0 = \rho_0.
    \end{align}
\end{subequations}
Next, we discretize this problem in both time and space. More precisely, discretizing over time into the time points $0, \Delta t, 2\Delta t, \ldots, 1$, where $\Delta t = 1/\ccT$, and over space into the grid points $x_1, \ldots, x_N$,
we get that $\ccP^v$ becomes a tensor $\bM\in \RR_+^{N^{\ccT + 1}}$, i.e., a nonnegative $(\ccT+1)$-mode tensor that represents the flow of the agents. More precisely, $\bM^{(i_0 \ldots i_\ccT)}$ is the discrete approximation corresponding to the probability of a sample path that passes through the discrete states
$x_{i_0}, \ldots, x_{i_\ccT}$ at the corresponding discrete time instances. Similarly, $\ccP^0$ becomes a nonnegative $(\ccT+1)$-mode tensor of probabilities corresponding to the evolution of the (uncontrolled) Wiener process. For reasons that will be clear shortly, we call this tensor $\bK$, and let $\bK^{(i_0 \ldots i_\ccT)} = \gamma \exp(-\bC^{(i_0 \ldots i_\ccT)}/\epsilon)$ for some tensor $\bC$ and where $\gamma > 0$ is a normalizing constant so that $\sum_{i_0, \ldots, i_\ccT} \bK^{(i_0 \ldots i_\ccT)} = 1$. The Kullback-Leibler divergence can then be discretized as
\begin{align*}
& \epsilon {\rm KL} (\ccP^v \| \ccP^0) \approx \epsilon \sum_{i_0, \ldots, i_\ccT} \log\left(\frac{\bM^{(i_0 \ldots i_\ccT)}}{ \bK^{(i_0 \ldots i_\ccT)}}\right) \bM^{(i_0 \ldots i_\ccT)} \\
 & =  \epsilon  \sum_{i_0, \ldots, i_\ccT}  \log\left(\bM^{(i_0 \ldots i_\ccT)}\right) \bM^{(i_0 \ldots i_\ccT)} - \epsilon  \sum_{i_0, \ldots, i_\ccT}  \log\left( \gamma \right) \bM^{(i_0 \ldots i_\ccT)} - \epsilon  \sum_{i_0, \ldots, i_\ccT}  \log\left((\exp(-\bC^{(i_0 \ldots i_\ccT)}/\epsilon)\right) \bM^{(i_0 \ldots i_\ccT)} \\
 & =  \epsilon D(\bM) + \text{constant}  - \epsilon  \sum_{i_0, \ldots, i_\ccT} \frac{-\bC^{(i_0 \ldots i_\ccT)}}{\epsilon} \bM^{(i_0 \ldots i_\ccT)} \\
  & =  \epsilon D(\bM) + \text{constant}  + \langle \bC, \bM \rangle.
\end{align*}
Discarding the constants, the discretized version of \eqref{eq:densitycontrol_KL_formulation}, and thus the discretized version of \eqref{eq:densitycontrol}, therefore becomes
\begin{subequations}\label{eq:discr_optimal_control}
    \begin{align}
    \minwrt[\substack{\bM\in \RR_+^{N^{\ccT + 1}}\\ \mu_1,\ldots,\mu_\ccT \in \RR_+^N}] & \; \langle \bC, \bM \rangle + \epsilon D(\bM) + \Delta t \! \sum_{j=1}^{\ccT-1} F_{j}(\mu_j) + G(\mu_\ccT) \label{eq:discr_optimal_control_cost} \\
    \text{subject to } \; & \; P_j(\bM) = \mu_j, \; j=1,2,\ldots, \ccT, \label{eq:discr_optimal_control_const1} \\
    & \; P_0(\bM) = \mu_0. \label{eq:discr_optimal_control_const2}
    \end{align}
\end{subequations}
Here, $\mu_0$ is a discrete approximation of $\rho_0$, and $\mu_j$ is the distribution of agents at time point $j$.
Left to do is to derive the form of the $(\ccT+1)$-mode tensor $\bC$, which in \eqref{eq:discr_optimal_control} can be seen to correspond to a cost of moving agents.
To this end, first note that since $\ccP^0$ is the probability distribution on path space corresponding to a time-homogeneous Markov process, we have that
\[
\bK^{(i_0 \ldots i_\ccT)} = \prod_{j = 0}^{\ccT - 1} K^{(i_j, i_{j+1})}
\]
where $K$ is a $N \times N$ matrix defining the transition probabilities between the discrete states in one time step. This in turn means that
the cost tensor takes the form
\begin{subequations}\label{eq:cost_tree}
    \begin{equation}
    \bC^{(i_0 \ldots i_\ccT)} = \sum_{j = 0}^{\ccT - 1} C^{(i_j, i_{j+1})},
    \end{equation}
where $C$ is a $N \times N$ matrix defining the transition costs between time points in \eqref{eq:discr_optimal_control}.
More precisely, we approximate the elements $C^{(i,k)}$ as the optimal cost of moving mass from discretization point $x_{i}$ to discretization point $x_{k}$ in one time step, given by
    \begin{equation}\label{eq:C_elem}
    C^{(i,k)} \; = \;
    \begin{cases}
    \displaystyle \; \minwrt[{v\in L_2([0,\Delta t])}] & \; \displaystyle \int_{0}^{\Delta t} \frac{1}{2} \| v \|^2 dt \\
    \; \text{subject to } & \; \displaystyle \dot{x} = f(x) + B(x)v \\
    & \; \displaystyle x(0) = x_{i}, \quad x(\Delta t) = x_{k}.
    \end{cases}
    \end{equation}
\end{subequations}%
This is approximation is motivated by the fact that for small time steps (corresponding to a small noise level in a time-rescaled version of the problem) there is a concentration of the probability around the trajectories that are solutions to the corresponding optimal control problem \cite{hijab1984asymptotic} (cf.~\cite[Sec.~5]{dai1991stochastic}, \cite[Sec.~IV]{chen2017optimal}, \cite[Thm.~2]{todorov2008general}).
Intuitively, this makes sense since for small time steps the transitions are approximately Gaussian. However, the ``distance'' in the transition is no longer measured in the Euclidean norm, but instead by the optimal control cost, since a lower control cost means that the system is ``easier'' to steer between the two states and hence it is more likely that the system makes that transition.

The optimal control problem \eqref{eq:C_elem} can typically not be solved analytically, except in the linear-quadratic case. Nevertheless, a numerical solution to the problem suffices, and the computation of the cost function $C$ can be done off-line before solving \eqref{eq:discr_optimal_control}.

Finally, note that the problem \eqref{eq:discr_optimal_control}--\eqref{eq:cost_tree} is a convex graph-structured tensor optimization problem of the form \eqref{eq:omt_multi_graph_convex_v2} on a path-graph.
In order to guarantee that \eqref{eq:discr_optimal_control} has an optimal solution, one must thus guarantee that it has a  feasible solution with finite object function value, i.e., that there is an $\bM \in \RR_+^{N^{\ccT + 1}}$ that fulfills \eqref{eq:discr_optimal_control_const1} and \eqref{eq:discr_optimal_control_const2} and is such that \eqref{eq:discr_optimal_control_cost} is finite (cf. Assumption~\ref{ass:primal_feasibility_and_optimality}).
A sufficient condition for this to hold is that the functions $F_{j}$, for $j=1, \ldots, \ccT-1$, and $G$ are finite on all of $\RR_+^N$, and that the elements \eqref{eq:C_elem} are all finite.
That latter is true if the deterministic counterpart to system \eqref{eq:SDE} is controllable in the (rather strong) sense that for all $x_0, x_1 \in X$ and for all $t > 0$ there exists a control signal in $L_{2}([0,t])$ that transitions the system from the initial state $x(0) = x_0$ to the final state $x(t) = x_1$. Two examples of classes of systems that have this property are controllable linear systems, and systems where $B(x)$ is square and invertible for all $x$.

\begin{remark}
Another solution method for solving problems of the form \eqref{eq:discr_optimal_control}, for agents that follow the dynamics of a first-order integrator, has been presented in \cite{BenCarDiNen19}. The two methods are similar, and the main difference is that the computational method developed in \cite{BenCarDiNen19} is based on a variable elimination technique, in contrast to the belief-propagation-type technique used here; see the discussion just before Theorem~\ref{thm:convergence}.
\end{remark}

\subsection{The multi-species problem}
A multi-species mean field game is an extension of mean field games to a set of heterogeneous agents, and the idea was already presented in the seminal work \cite{huang2006large, lasry2007mean}.
Here, we consider a multi-species potential mean field game which has $L$ different populations,
each of which can be associated with different costs and constraints,
and where each infinitesimal agent of species $\ell$ obeys the dynamics
\begin{equation*}%\label{eq:SDE_l}
dx_\ell(t) = f(x_\ell)dt + B(x_\ell)\big(v_\ell dt + \sqrt{\epsilon}dw_\ell\big),
\end{equation*}
with initial condition $x_\ell(0) = x_{\ell, 0} \sim \rho_{\ell, 0}(x)$.
Next, let $\rho_\ell(t, \cdot)$ denote the distribution of species $\ell$ at time point $t$, and 
note that a multi-species potential mean field game can, analogously to the single species game, be formulated as an optimal control problem over densities.
More precisely, the problem of interest here takes the form
\begin{subequations}\label{eq:densitycontrol_multispecies}
    \begin{align}
        \minwrt[\substack{\rho, \rho_\ell, v_\ell}] & \; \int_0^1 \!\!\! \int_X \sum_{\ell = 1}^L \frac{1}{2} \|v_\ell\|^2 \rho_\ell \, dx dt + \! \int_0^1 \!\! \mathcal{F}_t(\rho(t,\cdot)) dt  + \mathcal{G}(\rho(1,\cdot)) \nonumber \\
        & \quad + \sum_{\ell=1}^L \left( \int_0^1 \mathcal{F}_{\ell, t}(\rho_\ell(t, \cdot)) dt + \mathcal{G}_{\ell}(\rho_\ell(1, \cdot)) \right) \label{eq:densitycontrol_multispecies1} \\        
        \text{subject to } & \;  \frac{\partial \rho_\ell}{\partial t} + \nabla\cdot((f(x) + B(x)v_\ell) \rho_\ell ) \nonumber \\
        & \quad - 
\frac{\epsilon}{2} \sum_{i,k=1}^n \frac{\partial^2 (\sigma_{ik}  \rho_\ell)}{\partial x_i \partial x_k}
= 0, \quad \ell = 1, \ldots L, \label{eq:densitycontrol_multispecies2} \\
        & \; \rho_\ell(0,\cdot) = \rho_{\ell, 0}, \quad \rho(t,x) = \sum_{\ell=1}^L \rho_\ell(t,x), \label{eq:densitycontrol_multispecies3}	
    \end{align}
\end{subequations}
where we impose the same assumptions on $\mathcal{F}_{\ell, t}$ and $\mathcal{G}_{\ell}$ as on $\mathcal{F}_t$ and $\mathcal{G}$, respectively.
The functionals $\int_0^1 \mathcal{F}_t(\cdot) dt$ and $\mathcal{G}(\cdot)$ are the cooperative part of the cost, which connects the different species. In particular, for $\mathcal{F}_t \equiv 0$, $\mathcal{G} \equiv 0$, \eqref{eq:densitycontrol_multispecies} reduces to $L$ independent single-species problems.
Moreover, the functionals $\int_0^1 \mathcal{F}_{\ell, t}(\cdot) dt$ and $\mathcal{G}_{\ell}(\cdot)$ are the ones that give rise to the heterogeneity among the species.

\subsection{Numerical algorithm for solving the multi-species problem}\label{sec:num_alg_multispecies}
To derive a numerical algorithm for solving \eqref{eq:densitycontrol_multispecies}, analogously to the single-species problem we first discretize the problem over time and space.
To this end, by adapting the arguments in the previous section, we arrive at the discrete problem
\begin{subequations}\label{eq:multispecies_discrete}
    \begin{align}
    \minwrt[\substack{\bM_\ell, \, \mu_j, \, \mu_{\ell, j}\\ j = 1,\ldots, \ccT \\ \ell = 1, \ldots, L}] & \; \sum_{\ell = 1}^L \big( \langle \bC, \bM_\ell \rangle + \epsilon D(\bM_\ell) \big) + \Delta t \sum_{j=1}^{\ccT-1} F_{j}(\mu_j) + G(\mu_\ccT) \nonumber \\[-12pt]
    &\; + \sum_{\ell = 1}^L \left( \Delta t \sum_{j=1}^{\ccT-1} F_{\ell, j}( \mu_{\ell, j}) + G_{\ell}(\mu_{\ell, \ccT}) \right) \label{eq:multispecies_discrete1}\\
    \text{subject to } & \; P_j(\bM_\ell) = \mu_{\ell, j},\; j=1, \ldots, \ccT, \, \ell = 1, \ldots, L, \label{eq:multispecies_discrete2} \\
    & \; P_0(\bM_\ell) = \mu_{\ell, 0}, \; \ell = 1, \ldots, L, \label{eq:multispecies_discrete3} \\
    & \; \sum_{\ell = 1}^L \mu_{\ell, j} = \mu_j, \; j=0, \ldots, \ccT \label{eq:multispecies_discrete4}
    \end{align}
\end{subequations}
where $\bC$ still has the form \eqref{eq:cost_tree}, and where $\mu_{\ell, 0}$ are discrete approximations of $\rho_{\ell, 0}$.  In particular, note that the second line in the cost \eqref{eq:multispecies_discrete1} is the discretization of the second line in \eqref{eq:densitycontrol_multispecies1}.
Moreover, also note that \eqref{eq:multispecies_discrete} consists of $L$ coupled graph-structured tensor optimization problems, coupled via the constraint \eqref{eq:multispecies_discrete4} and the cost imposed on $\mu_j$, for $j = 1,\ldots, \ccT$, in \eqref{eq:multispecies_discrete1}.

Next, we reformulate \eqref{eq:multispecies_discrete} into one single graph-structured tensor optimization problem (cf.~\cite{haasler2021scalable}).
To this end, let $\bM \in \RR^{L\times N^{\ccT+1}}$ be the $(\ccT+2)$-mode tensor such that $\bM^{(\ell i_{0} \dots i_\ccT)} = (\bM_\ell)^{(i_{0} \dots i_\ccT)}$, i.e., $\bM^{(\ell i_{0} \dots i_\ccT)}$ is the amount of mass of species $\ell$ that moves along the path $x_{i_{0}}, \dots, x_{i_\ccT}$. For this tensor $\bM$, we will use the index $-1$ to denote the ``species index''. This  means that $(P_{-1}(\bM))^{(\ell)} = \sum_{i_0, \ldots, i_\ccT} (\bM_\ell)^{(i_{0} \dots i_\ccT)}$, for $\ell = 1, \ldots, L$, and hence the elements of the additional marginal $\mu_{-1} \in \RR_+^L$ are the total mass of the densities of the different species. Moreover, this means that 
$P_j(\bM)$ is the total distribution $\mu_j$ at time $j\Delta t$, as defined by \eqref{eq:multispecies_discrete4}, while the bimarginal projection $P_{-1, j}(\bM)$ gives the $L \times N$ matrix $[\mu_{1,j}, \ldots, \mu_{L,j}]^T$.
By introducing
\begin{align*}
& \SpeciesMtx_{-1,0}=[\SpeciesMtxElem_{1,0},\ldots, \SpeciesMtxElem_{L,0}]^T\in \RR_+^{L\times N},
\end{align*}
the constraint
\eqref{eq:multispecies_discrete3} can be imposed by requiring that $P_{-1, 0}(\bM) = \SpeciesMtx_{-1,0}$. 
Next, by defining the functions $\mathscr{F}_j^{L} : \RR^{L \times N} \to \RR$ as
\[
\mathscr{F}_j^{L}(R_{-1, j}) = \sum_{\ell = 1}^L \Delta t F_{\ell, j} ( \mu_{\ell, j}) \quad j = 1, \ldots, \ccT,
\]
and similarly for $\mathscr{G}^{L}$,
the last term in the cost \eqref{eq:multispecies_discrete1} can be written as functionals applied to the bimarginal projections.
Finally, by noting that $\sum_{\ell = 1}^L D(\bM_\ell) = D(\bM)$, we can write the problem as
\begin{subequations}\label{eq:multispecies_discrete_rewritten}
    \begin{align}
    \minwrt[\substack{\bM, \, \mu_j, \, R_{-1, j}\\ j=1,\ldots, \ccT}] & \; \langle \newC, \bM \rangle + \epsilon D(\bM) + \Delta t \sum_{j=1}^{\ccT-1} F_{j}(\mu_j) + G(\mu_\ccT) \nonumber \\[-10pt]
    & \; + \sum_{j=1}^{\ccT-1} \mathscr{F}_j^{L}( R_{-1, j} ) + \mathscr{G}^{L}(R_{-1, \ccT}) \label{eq:multispecies_discrete_rewritten1}\\
    \text{subject to } & \; P_j(\bM) = \mu_j,\quad j=1,\ldots, \ccT, \label{eq:multispecies_discrete_rewritten2} \\
    & \; P_{-1,j}(\bM) = R_{-1, j}, \quad j=1,\ldots, \ccT, \label{eq:multispecies_discrete_rewritten4} \\
    & \; P_{-1,0}(\bM) = \SpeciesMtx_{-1, 0} \label{eq:multispecies_discrete_rewritten3}
    \end{align}
where
\begin{equation}\label{eq:multispecies_discrete_rewritten_cost_tree}
\newC^{(\ell i_0 \ldots i_\ccT)} = \sum_{j = 0}^{\ccT - 1} C^{(i_{j}, i_{j+1})}.
\end{equation}
\end{subequations}

The problem \eqref{eq:multispecies_discrete_rewritten} is readily seen to be a graph-structured tensor optimization problem of the form \eqref{eq:omt_multi_graph_convex_v2}, and can hence be solved using Algorithm~\ref{alg:generalized_sinkhorn}.
In particular, the iterates of the transport plan produced by Algorithm~\ref{alg:generalized_sinkhorn} are of the form $\bM^k = \bK \odot \bU^k$,
where $\bK = \exp(-\newC/\epsilon)$ and where
\begin{equation}\label{eq:U_multi_species_new}
\bU^{(\ell i_{0} \dots i_\ccT)} = U_{-1,0}^{(\ell, i_{0})} \prod_{j = 1}^\ccT U_{-1,j}^{(\ell, i_{j})}  \prod_{j = 1}^\ccT u_j^{(i_j)}.
\end{equation}
The underlying graph-structure is illustrated in Figure~\ref{fig:multispecies_graph_nonfixed}, and by adapting the arguments in \cite{haasler2021scalable}, marginal and bimarginal projections needed in the inclusion problems \eqref{eq:sinkhorn_inclusion_u} can be computed efficiently as follows.
\begin{theorem}\label{thm:proj}
	Let $\bK=\exp(-\newC/\epsilon)$, with $\newC$ defined as in \eqref{eq:multispecies_discrete_rewritten_cost_tree} and $\epsilon>0$, and let $\bU$ be as in \eqref{eq:U_multi_species_new}. 
	Define $K=\exp(-C/\epsilon)$, and let
	\begin{equation*}%\label{eq:psi_hat}
	\hat \Psi_j = \begin{cases}
	 U_{-1,0}  K , & j=1,\\
	\left( \hat \Psi_{j-1} \odot U_{-1, j-1} \right) \diag(u_{j-1}) K  , & j= 2,\dots,\ccT,
	\end{cases} 
	\end{equation*}
	and
	\begin{equation*}%\label{eq:psi}
	\Psi_j = \begin{cases}
	U_{-1, \ccT} \, \diag(u_\ccT) K^T, & j=\ccT-1,\\
	\left( \Psi_{j+1} \odot U_{-1, j+1} \right) \diag(u_{j+1}) K^T , & j=0,\dots, \ccT-2.
	\end{cases}
	\end{equation*}
	Then we have the following expressions for projections of the tensor $\bK \odot \bU$
	\begin{align*}
	& P_{-1,0} (\bK \odot \bU) = U_{-1,0} \odot  \Psi_0, \\
	& P_{-1,j} (\bK \odot \bU) = \hat{\Psi}_{j} \odot  \Psi_j \odot U_{-1, j}\diag(u_j),  \\
	& P_{-1,\ccT} (\bK \odot \bU) = U_{-1, \ccT}\diag(u_\ccT) \odot \hat{\Psi}_{\ccT}, \\
	& P_\ccT(\bK \odot \bU) =  u_\ccT \odot  \left(   \hat \Psi_{\ccT} \odot U_{-1, \ccT}  \right)^T \ett,\\	
	& P_j(\bK \odot \bU) = u_j \odot  \left(   \hat \Psi_{j} \odot   \Psi_j \odot U_{-1,j}  \right)^T \ett,
	\end{align*}
	for $j=1,\dots,\ccT-1$.
\end{theorem}

\begin{proof}
See Appendix~\ref{app:proofs}.
\end{proof}

Finally, using Theorem~\ref{thm:proj} and specializing Algorithm~\ref{alg:generalized_sinkhorn} to solving the particular problem \eqref{eq:multispecies_discrete_rewritten}, an algorithm for solving discretized multi-species potential mean field games
is given in Algorithm~\ref{alg:multi_species}.
\begin{remark}\label{rem:meanfieldgames}
The algorithms in \cite{ringh2021efficient} are special instances of Algorithm~\ref{alg:multi_species}. In particular, if $\mathscr{F}_j^{L} (\cdot) = \langle C_j, \cdot \rangle$ for some $C_j\in \RR^{L \times N}$, then $(\mathscr{F}_j^{L})^*(\cdot) = \indFun_{\{ C_j \}}(\cdot)$.
Hence, $U_{-1, j}$ must equal $K_j := \exp(-C_j/\epsilon)$. Similarly, if $\mathscr{G}^{L}(\cdot) = \langle C_\ccT, \cdot \rangle$, we get that $U_{-j, \ccT}$ must be equal to $K_\ccT$, from which we recover \cite[Alg.~1]{ringh2021efficient}. On the other hand, if $\mathscr{G}^{L}(\cdot) = \indFun_{\{\SpeciesMtx_{-1, \ccT} \}}(\cdot)$ for some given $\SpeciesMtx_{-1, \ccT}$, then the marginal $\mu_\ccT$ is also known and any cost associated with it is a constant and can hence be removed. Moreover, $(\mathscr{G}^{L})^*(\cdot) = \langle \SpeciesMtx_{-1, \ccT}, \cdot \rangle$, from which we recover \cite[Alg.~2]{ringh2021efficient}.
\end{remark}

\begin{figure}
  \begin{center}
	\begin{tikzpicture}[scale=0.95, every node/.style={scale=0.95}]
	\footnotesize
	
	\tikzstyle{main}=[circle, minimum size = 9mm, thick, draw =black!80, node distance = 10mm]
	\tikzstyle{obs}=[circle, minimum size = 9mm, thick, draw =black!80, node distance = 10 mm and 6mm ]
	
	\node[main,fill=black!10] (mu0) {$\mu_{-1}$};
	
	\node[] (phi1c) [below=of mu0] {}; 
	\node[] (mu3) [left=of phi1c] {};  
	\node[obs] (mu2) [left=of mu3] {$\mu_{1}$};  
	\node[obs,fill=black!10] (mu1) [left=of mu2] {$\mu_{0}$};
	\node[] (muS2) [right=of phi1c] {};	
	\node[obs] (muS1) [right=of muS2] {$\!\mu_{\!\ccT\!-\!1\!}\!$};
	\node[obs] (muS) [right=of muS1] {$\mu_{\ccT}$};
	
	\draw[->, -latex, thick] (mu0) -- node[above left] {$P_{-1,0}(\bM)=\SpeciesMtx_{-1,0}$} (mu1);
	\draw[->, -latex, thick] (mu0) -- node[below right] {$R_{-1, 1}$} (mu2);
	\draw[->, -latex, thick] (mu0) -- node[below left] {\ $R_{-1,\ccT-1}$} (muS1);
	\draw[->, -latex, thick] (mu0) -- node[above right] {\ $R_{-1,\ccT}$} (muS);
	
	\draw[->, -latex, thick] (mu1) -- node[below] {$C$} (mu2);
    \draw[->, -latex, thick] (mu2) -- node[below] {$C$} (mu3);
	\draw[->, -latex, thick] (muS1) -- node[below] {$C$} (muS);
	\draw[->, -latex, thick] (muS2) -- node[below] {$C$} (muS1);
	
	\draw[loosely dotted, very thick] (mu3) -- (muS2); 
	
	\end{tikzpicture}
	\caption{Illustration of the graph $\ccG$ for the multi-species density optimal control problem. Grey circles correspond to known densities, and white circles correspond to densities which are to be optimized over.}	
	\label{fig:multispecies_graph_nonfixed}
  \end{center}
\end{figure}
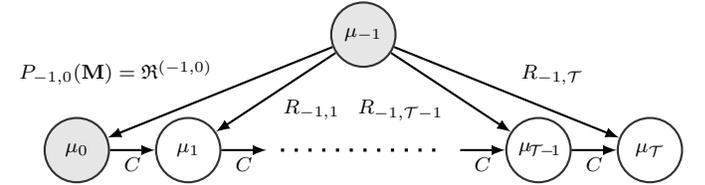

\begin{algorithm}[tb]
	\begin{algorithmic}[1]
	    \STATE Given: Initial guess $u_1,\dots, u_{\ccT}$, $U_{-1,0}, \ldots, U_{-1,\ccT}$
		\WHILE{Not converged}	
		\STATE $ \Psi_{\ccT-1} \leftarrow U_{-1, \ccT} \, \diag(u_\ccT) K^T$ 
		\FOR{ $j=\ccT-2,\dots, 0$}
		\STATE  $\Psi_{j} \leftarrow ( \Psi_{j+1} \odot U_{-1, j+1} ) \diag(u_{j+1}) K^T $
		\ENDFOR 
		\STATE $U_{-1,0} \leftarrow \SpeciesMtx_{-1,0}\oslash \Psi_0$
		\STATE $\hat \Psi_1 \leftarrow U_{-1,0} K$ 
		\FOR{ $j=1,\dots,\ccT-1$}
		\STATE $W_{-1, j} \leftarrow (\hat \Psi_{j} \odot   \Psi_j)\diag(u_j)$, and update $U_{-1,j}$ by solving \eqref{eq:sinkhorn_inclusion_bimarginal_U}
		\STATE $w_{j} \leftarrow (\hat \Psi_{j} \odot   \Psi_j \odot U_{-1,j})^T \ett$, and update $u_j$ by solving \eqref{eq:sinkhorn_inclusion_marginal_u}
		\STATE $\hat \Psi_{j+1} \leftarrow ( \hat \Psi_{j} \odot U_{-1,j} ) \diag(u_{j}) K $ 
		\ENDFOR
		\STATE $W_{-1, \ccT} \leftarrow \hat \Psi_{\ccT} \diag(u_\ccT)$, and update $U_{-1,\ccT}$  by solving \eqref{eq:sinkhorn_inclusion_bimarginal_U}
		\STATE $w_{\ccT} \leftarrow  (\hat \Psi_{\ccT} \odot U_{-1, \ccT})^T \ett$, and update $u_\ccT$ by solving \eqref{eq:sinkhorn_inclusion_marginal_u}
		\ENDWHILE
	  \RETURN $u_1,\dots,u_{\ccT}$, $U_{-1,0}, \ldots, U_{-1,\ccT}$
	\end{algorithmic}
	\caption{Method for solving the multi-species potential mean field game \eqref{eq:multispecies_discrete_rewritten}.}
	\label{alg:multi_species}
\end{algorithm}

\subsection{Numerical example}\label{sec:2D_example}
In this section we demonstrate Algorithm~\ref{alg:multi_species} on a two-dimensional numerical example with $L = 4$ different species. To this end, we consider the state space $[0,3] \times [0,3]$ and uniformly discretize it into $100 \times 100$ grid points; the latter are denoted $x_{i,k}$ for $i,k = 1, \ldots, 100$. No points are placed on the boundary of the state space, which means that the cell size is $\Delta x = 0.03^2$. Moreover, time is discretized into $\ccT+1 = 40$ time steps, i.e., with a discretization size $\Delta t = 1/39$ and with time index $j = 0, \ldots, 39$. The dynamics of the agents is taken to be
$f(x) \equiv 0$ and $B(x) = I$.
This means that the cost matrix, with elements \eqref{eq:C_elem},
is time-independent and given by $C = [\|x_{i_1, k_1} - x_{i_2, k_2}\|^2]_{i_1, i_2, k_1, k_2 = 1}^{100}$. This corresponds to the squared Wasserstein-2 distance on the discrete grid.

For $\epsilon = 10^{-2}$, we consider the discrete problem
\begin{subequations}\label{eq:2d_example}
  \begin{align}  
    \minwrt[\substack{\bM_\ell \in \RR_+^{(100^2)^{40}},\\ \mu_{\ell, j} \in \RR_+^{100^2}\\ j=1,\ldots, 39,\; \ell=1,2,3,4}]
        &  \quad \sum_{\ell = 1}^4 \Big( \langle \bC, \bM_\ell \rangle + \epsilon D(\bM_\ell) \Big) + \sum_{j=1}^{39}\langle c_{3}, \mu_{3, j} \rangle \nonumber \\[-12pt]
    & \quad + 0.1\sum_{j = 1}^{39} \| \mu_{4, j} - \tilde{\nu} \|_2^2 + 3 \| \mu_{19} - \tilde{\mu}_1 \|^2_2 + 3 \| \mu_{39} - \tilde{\mu}_2 \|^2_2 \label{eq:2d_example_cost} \\
    \text{subject to } \quad \; & \quad P_j(\bM_\ell) = \mu_{\ell, j}, \quad j=0, \ldots, 39, \; \ell = 1, 2, 3, 4, \\
    & \quad \sum_{\ell = 1}^4\mu_{\ell, j} = \mu_j, \quad j= 0,\ldots, 39, \\
    & \quad \mu_j \leq \kappa_j, \quad j= 1,\ldots, 39, \label{eq:2d_example_const3}\\
    & \quad \mu_{1, j} \leq \tilde{\kappa}, \quad j= 1,\ldots, 39. \label{eq:2d_example_const4}
  \end{align}
\end{subequations}
Here, $\tilde{\mu}_1$ and $\tilde{\mu}_2$ are the two distributions given in Figure~\ref{subfig:2d_ex_target_dists}.
Moreover, the  linear cost $c_{3}$, associated with species $3$, and the target distribution $\tilde{\nu}$, associated with species $4$, are both given in Figure~\ref{subfig:2d_ex_linear_costs}.%
\footnote{Note that $\tilde{\mu}_2$ and $\tilde{\nu}$ are uniform distributions. The former has the same total mass as the total distribution $\mu_0$, and the latter the same as $\mu_{4, 0}$.}
Finally, for the capacity constraint \eqref{eq:2d_example_const3}, $\kappa_j$ is illustrated in Figure~\ref{subfig:2d_ex_constraints}, while for the capacity constraint \eqref{eq:2d_example_const4}, $\tilde{\kappa}$ is zero in the lower half of the domain and infinite for the upper half.

The graph-structured tensor optimization reformulation of \eqref{eq:2d_example} was solved using Algorithm~\ref{alg:multi_species}. The latter is adapted as in Section~\ref{sec:multiple_costs} to handle both the costs on the total marginals in \eqref{eq:2d_example_cost} and the inequality constraints in \eqref{eq:2d_example_const3}; details on the Fenchel conjugates of the functions involved can be found in Appendix~\ref{app:fenchel}.
Results are shown in Figure~\ref{fig:2d_ex_species}, where the initial distributions $\mu_{\ell, 0}$ for the different agents can be seen in the left-most column (showing time point $j = 0$).

\begin{figure}[tbh]
  \begin{center}
    \begin{subfigure}[t]{.32\textwidth}
      \centering
      \includegraphics[trim=1.3cm 0.3cm 1cm 0cm, clip=true,width=\textwidth]{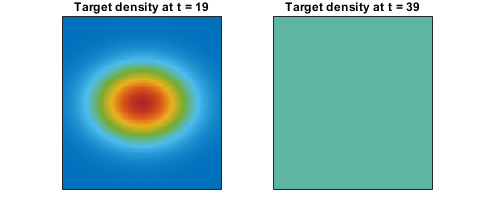}
      \subcaption{Target densities for total density.}
      \label{subfig:2d_ex_target_dists}
    \end{subfigure}
    \hfill
    \begin{subfigure}[t]{.32\textwidth}
      \centering
      \includegraphics[trim=2cm 0.3cm 1.3cm 0cm, clip=true, width=\textwidth]{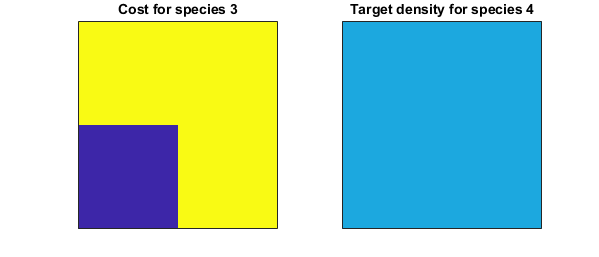}
      \subcaption{Species-dependent cost and constraint}
      \label{subfig:2d_ex_linear_costs}
    \end{subfigure}
    \hfill
    \begin{subfigure}[t]{.32\textwidth}
      \centering
      \includegraphics[trim=6.2cm 1.65cm 4.5cm 1.5cm, clip=true, width=\textwidth]{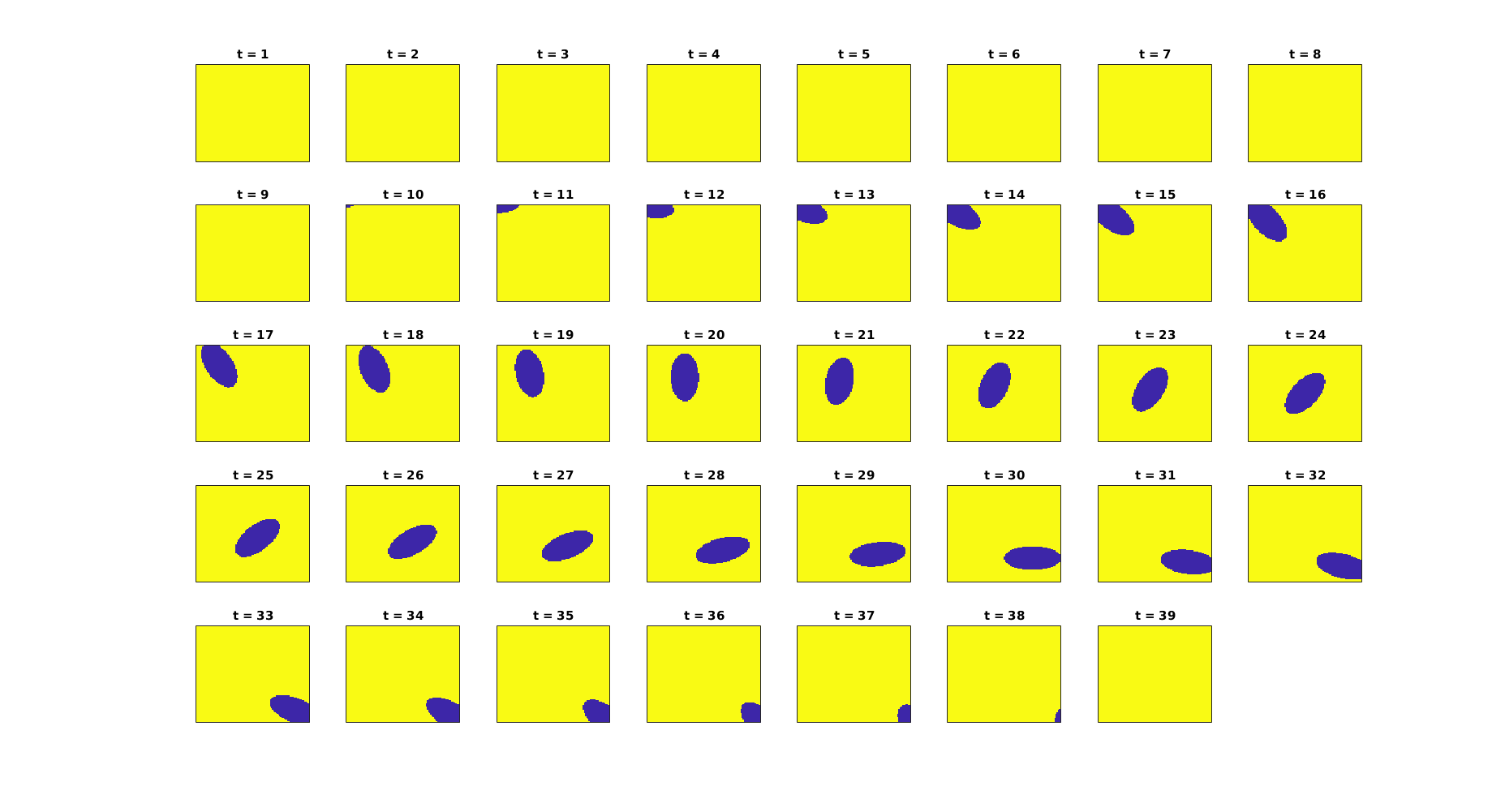}
      \subcaption{Illustration of the capacity constraints.}
      \label{subfig:2d_ex_constraints}
    \end{subfigure}
    \begin{subfigure}[t]{\textwidth}
      \centering
      \includegraphics[trim=1.4cm 0.3cm 1.6cm 0cm, clip=true, width=\textwidth]{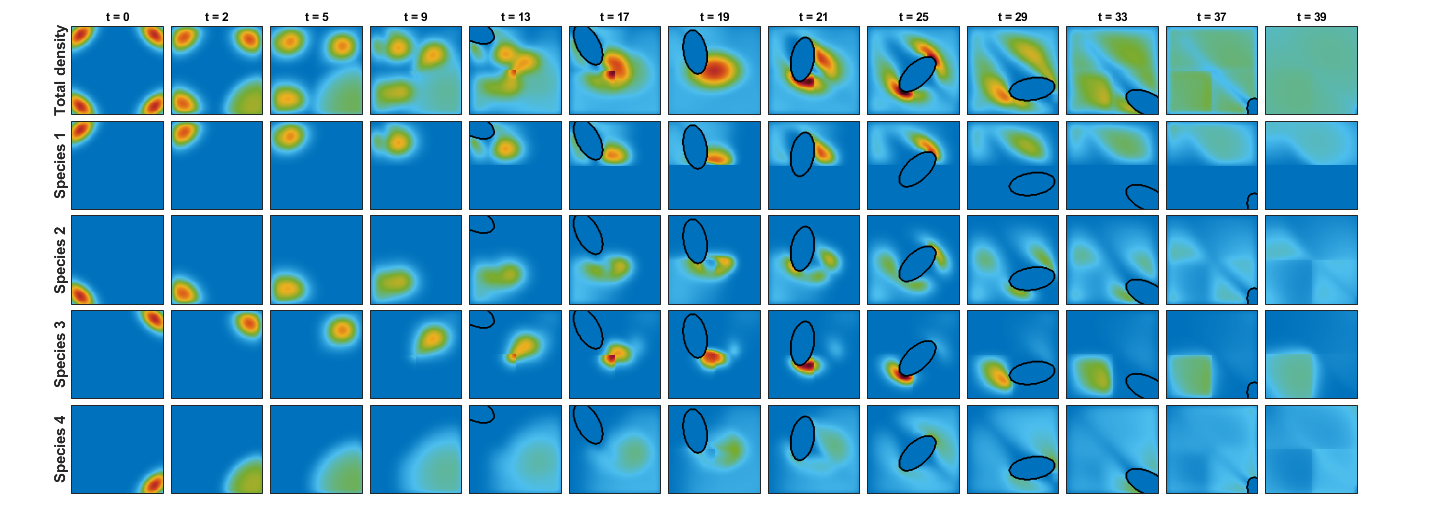}
      \caption{Optimal solution.}
      \label{fig:2d_ex_species}
    \end{subfigure}
    \caption{Figures describing the setup in the numerical example in Section~\ref{sec:2D_example}. (a)~Target densities $\tilde{\mu}_1$ (left) and $\tilde{\mu}_2$ (right) for the total density at time points $j=19$ and $j=39$, respectively. (b)~Illustration of species-dependent cost and constraint: left plot shows the linear cost $c_3$ for species $3$, where blue means cost 0 and yellow means a cost of $390\Delta x\Delta t$. The right plots shows the target distributions $\tilde{\nu}$ for species~$4$. (c)~The capacity constraint $\kappa_j$ at the different time points $j$: blue means zero capacity (obstacle) while yellow means infinite capacity. (d)~The optimal solution, illustrated as time evolution of total density and densities of the individual species.}
    \label{fig:2d_ex_rest}
  \end{center}
  
\end{figure}

\section{Conclusions}\label{sec:conclusions}

In this paper we have seen that graph structured tensor optimization problems naturally appear in several areas in systems and control. We have developed numerical algorithms for these problems based on dual coordinate ascent that utilize the fact that the dual problems decouple according to the graph structure. We also showed that under mild conditions these algorithms are globally convergent, and in certain cases the convergence is R-linear.
This framework can also be used to solve convex multi-commodity dynamic network flow problems akin to the ones studied in \cite{haasler2021scalable}. Moreover,
we believe that these methods are useful for addressing many other types of problems, e.g., in flow problems where the nodes or edges also have dynamics (cf.~\cite{como2017resilient}).
Moreover, we also believe that these methods can be extended to handle, e.g., multi-species potential mean field games where each species also has different dynamics.

\appendix

\section{Deferred proofs}\label{app:proofs}

\begin{proof}[Proof of Lemma~\ref{lem:primal_optimality}]
By Assumption~\ref{ass:primal_feasibility_and_optimality}, there is a feasible point to \eqref{eq:omt_multi_graph_convex_v2} with finite objective function value, and since problems \eqref{eq:omt_multi_graph_convex_v2} and \eqref{eq:omt_multi_graph_convex} are equivalent, this means 
that the objective function in \eqref{eq:omt_multi_graph_convex} is proper.
To show that the minimum for the latter is attained, note that $g_{t}$, $t\in \ccV$, and $f_{t_1,t_2}$, $(t_1,t_2)\in \ccE$, are all proper, convex, and lower-semicontinuous, and hence they all have a continuous affine minorant \cite[Thm.~9.20]{bauschke2017convex}.
However, since the entropy term $\epsilon D(\bM)$ is radially unbounded and grows faster towards $\infty$ than linearly, we therefore have that the objective function in \eqref{eq:omt_multi_graph_convex} is radially unbounded. Since the entire objective function is also proper, convex, and lower-semicontinuous, the minimum is attained \cite[Thm.~27.2]{rockafellar1970convex}, and it is unique since $D(\bM)$ (and hence the entire objective function in \eqref{eq:omt_multi_graph_convex}) is strictly convex.
\end{proof}

\begin{lemma}\label{lem:nonempty_realative_interior}
Let $f : \RR^n \to \RRext$ be proper, convex, and lower-semicontinuous, then $\ri(\dom(f^*)) \neq \emptyset$.
\end{lemma}

\begin{proof}
Since the function $f$ is proper, convex, and lower-semicontinuous, so is the function $f^*$ \cite[Cor.~13.38]{bauschke2017convex}. $\dom(f^*)$ is therefore nonempty, and by \cite[Prop.~8.2]{bauschke2017convex} it is convex. Using \cite[Fact~6.14(i)]{bauschke2017convex}, the result follows.
\end{proof}

\begin{lemma}\label{lem:strong_duality}
There is no duality gap between \eqref{eq:omt_multi_graph_convex_v2} and
\eqref{eq:dual}.
\end{lemma}

\begin{proof}
To prove it, we derive a Lagrangian dual to an equivalent problem to \eqref{eq:dual}, and show that for the latter strong duality holds with \eqref{eq:omt_multi_graph_convex_v2}. To this end, note that a problem with the same set of globally optimal solutions as \eqref{eq:dual} is the constrained optimization problem
\begin{align*}
\supwrt[\bU, \lambda, \Lambda]
& \; -\epsilon \langle \bK, \bU \rangle - \sum_{t\in \ccV}(g_{t})^*(- \lambda_t)
- \!\!\!\! \sum_{(t_1,t_2)\in \ccE}(f_{t_1,t_2})^*(- \Lambda_{t_1, t_2}) \\
\text{s.t. } & \quad \log(\bU^{(i_1 \ldots i_{\ccT})}) = \frac{1}{\epsilon} \left(\sum_{t\in \ccV} \lambda_t^{(i_t)}
 + \! \sum_{(t_1,t_2)\in \ccE}  \Lambda_{t_1, t_2}^{(i_{t_1}, i_{t_2})} \right).
\end{align*}
However, the latter is nonconvex due to the nonaffine equality constraint. Nevertheless, since $\bK \geq 0$ the cost function is nonincreasing in $\bU$, and since the logarithm is a monotone increasing function, the above problem has the same globally optimal solution as the relaxed, convex problem with the equality changed for an inequality $\geq$.
Moreover, for this convex problem, by using Lemma~\ref{lem:nonempty_realative_interior} it is easily seen that Slater's condition is fulfilled, and hence strong duality holds.
Next, relaxing the convex inequality constraints with multipliers $\mathbf{Q}^{(i_1 \ldots i_{\ccT})} \geq 0$ we get the Lagrangian
\begin{align*}
& -\epsilon \langle \bK, \bU \rangle - \sum_{t\in \ccV}(g_{t})^*(- \lambda_t)  - \sum_{(t_1,t_2)\in \ccE}(f_{t_1,t_2})^*(- \Lambda_{t_1, t_2}) \\
& + \!\! \sum_{i_1 \ldots i_{\ccT}} \! \mathbf{Q}^{(i_1 \ldots i_{\ccT})} \!\! \left( \! \log(\bU^{(i_1 \ldots i_{\ccT})} ) -  \frac{1}{\epsilon} \! \left(\sum_{t\in \ccV} \lambda_t^{(i_t)} + \!\!\!\! \sum_{(t_1,t_2)\in \ccE} \!\!\!\! \Lambda_{t_1, t_2}^{(i_{t_1}, i_{t_2})} \! \right)  \!\!  \right)
\end{align*}
which separates over $\lambda_t$, $\Lambda_{t_1, t_2}$, and $\bU$. Moreover, we have that $\sum_{i_1 \ldots i_{\ccT}} \! \mathbf{Q}^{(i_1 \ldots i_{\ccT})} \frac{1}{\epsilon} \lambda_t^{(i_t)} \! = \langle 1/\epsilon P_t(\mathbf{Q}), \lambda_t \rangle$, and therefore when taking the supremum over $\lambda_t$ we get
\begin{align*}
\sup_{\lambda_t \in \RR^N} - (g_{t})^*(- \lambda_t) - \langle 1/\epsilon P_t(\mathbf{Q}), \lambda_t \rangle & = (g_{t})^{**}(1/\epsilon P_t(\mathbf{Q}))  = g_{t}(1/\epsilon P_t(\mathbf{Q})),
\end{align*}
where the last equality follows from \cite[Thm.~13.37]{bauschke2017convex}; 
an analogous result holds for $(f_{t_1,t_2})^*$ and $\Lambda_{t_1, t_2}$. The remaining part of the Lagrangian is
$\sup_{\bU \in \RR^{N^\ccT}} -\epsilon \langle \bK, \bU \rangle + \langle \mathbf{Q}, \log(\bU) \rangle$, 
and to find this supremum
we first note that if $\bK^{(i_1 \ldots i_{\ccT})} = 0$, then we must have $\mathbf{Q}^{(i_1 \ldots i_{\ccT})} = 0$ or else the cost function is unbounded. For all other elements, we take the derivative with respect to $\bU^{(i_1 \ldots i_{\ccT})}$ and set it equal to zero,
from which it follows that $\bU^{(i_1 \ldots i_{\ccT})} = \mathbf{Q}^{(i_1 \ldots i_{\ccT})}/(\epsilon \bK^{(i_1 \ldots i_{\ccT})}) > 0$, which is hence the supremum.
Plugging this back into the cost, we get
\begin{align*}
 -\epsilon \langle \bK, \bU \rangle + \langle \mathbf{Q}, \log(\bU) \rangle & =  \sum_{i_1 \ldots i_{\ccT}} -\mathbf{Q}^{(i_1 \ldots i_{\ccT})} + \langle \mathbf{Q}, \log(\mathbf{Q}) \rangle - \langle \mathbf{Q}, \log(\epsilon \bK) \rangle \\
& = \sum_{i_1 \ldots i_{\ccT}} -\mathbf{Q}^{(i_1 \ldots i_{\ccT})} + \langle \mathbf{Q}, \log(\mathbf{Q}) - \log(\epsilon) \rangle + (1/\epsilon) \langle \mathbf{Q}, \bC \rangle,
\end{align*}
together with the constraints that $\mathbf{Q}^{(i_1 \ldots i_{\ccT})} = 0$ if $\bK^{(i_1 \ldots i_{\ccT})} = 0$. But for any element such that $\bK^{(i_1 \ldots i_{\ccT})} = 0$ we have that $\bC^{(i_1 \ldots i_{\ccT})} = \infty$, and the constraints can thus be removed since they are implicitly enforced by the cost function.
Therefore, with the change of variable $\mathbf{Q} = \epsilon\bM$ we recover, up to a constant, the primal problem \eqref{eq:omt_multi_graph_convex}. Since \eqref{eq:omt_multi_graph_convex} has the same optimal value as \eqref{eq:omt_multi_graph_convex_v2}, the result follows.
\end{proof}

\begin{proof}[Proof of Theorem~\ref{thm:proj}]
Note that
$\bK^{(\ell i_0 \ldots i_\ccT)} = \prod_{t = 0}^{\ccT-1} K^{(i_{t}, i_{t+1})}$.
Together with \eqref{eq:U_multi_species_new}, this means that
\begin{align*}
(P_{-1, j}(\bK \odot \bU))^{(\ell, i_j)} & = \sum_{\substack{i_0, \ldots i_{j-1} \\ i_{j+1}, \ldots i_{\ccT}}} \Bigg( \left(\prod_{t = 0}^{\ccT-1} K^{(i_{t}, i_{t+1})} U_{-1,0}^{(\ell, i_{0})}\right)  \left(\prod_{t = 1}^\ccT U_{-1,t}^{(\ell, i_{t})}\right)  \left(\prod_{ t=1}^\ccT u_t^{(i_t)}\right) \Bigg)
\\
& = U_{-1,j}^{(\ell, i_{j})}u_j^{(i_j)} \hat{\Psi}_j^{(\ell, j)} \Psi_j^{(\ell, j)},
\end{align*}
where
\begin{align*}
\hat{\Psi}_j^{(\ell, i_{j})} & = \sum_{i_0, \ldots i_{j-1}} U_{-1,0}^{(\ell, i_{0})} K^{(i_{0}, i_{1})} \prod_{t = 1}^{j-1} U_{-1,t}^{(\ell, i_{t})}u_t^{(i_t)} K^{(i_{t}, i_{t+1})}, \\
\Psi_j^{(\ell, i_{j})} & = \! \sum_{i_{j+1}, \ldots i_{\ccT}} \! U_{-1,\ccT}^{(\ell, i_{\ccT})} u_\ccT^{(i_\ccT)} K^{(i_{\ccT-1}, i_{\ccT})} \! \prod_{t = j+1}^{\ccT-1} \! U_{-1,t}^{(\ell, i_{t})} u_t^{(i_t)} K^{(i_{t-1}, i_{t})}.
\end{align*}
A direct calculation gives that $\hat{\Psi}_j$ and $\Psi_j$ above fulfill the recursive definitions in the theorem, which proves the form of the bimarginal projection for $j = 1, \ldots, \ccT-1$. Next, the form of the bimarginal projections for $j = 0$ and $\ccT$ can be readily verified analogously. Finally, note that
$
(P_j(\bK \odot \bU))^{(i_j)} = \sum_{\ell=1}^L (P_{-1, j}(\bK \odot \bM))^{(\ell, i_j)},
$
which gives the result for the projections and proves the theorem.
\end{proof}

\section{Fenchel conjugates of some functions}\label{app:fenchel}
In all the examples below, let $f : \RR^n \to \RRext$.

\begin{example}
Let $\alpha, \beta \in \RRext$, $\alpha_i \leq \beta_i$ for $i = 1, \ldots, n$, and $[\alpha, \beta] := \{ y \in \RR^n \mid \alpha^{(i)} \leq y^{(i)} \leq \beta^{(i)}, i = 1, \ldots, n\}$. For a set $A \subset \RR$, let $\mathcal{I}_A$ be the characteristic function $\mathcal{I}_A(x) = 1$ if $x \in A$ and $0$ else. The Fenchel conjugate of 
$
f(x) = \indFun_{[\alpha, \beta]}(x)$ is $f^*(x^*) = \sum_{i = 1}^n \Big(( x^* )^{(i)} \beta^{(i)} \mathcal{I}_{\RR_+}( (x^*)^{(i)}) + (x^*)^{(i)} \alpha^{(i)} \mathcal{I}_{\RR_-}( (x^*)^{(i)}) \Big).
$
\end{example}

\begin{example}
Let $p \in (1, \infty)$, let $\sigma > 0$, and let $y \in \RR^n$. The Fenchel conjugate of $ f(x) = \sigma \| x - y \|_p^p$ is $ f^*(x^*) =  \langle x^*,y \rangle + \frac{1}{q \, \sigma^{q-1} \, p^{q-1}} \| x^* \|_q^q$, where $1/p + 1/q = 1$.
\end{example}

\begin{example}
Let $\beta \in \RR^n$ and let $\beta_i > 0$ for $i = 1, \ldots, n$. The Fenchel conjugate of
$f(x) = x\oslash(\beta - x) + \indFun_{[0, \beta]}(x)$ is $f^*(x^*) = \sum_{i=1}^n f^*_i(x^*_i)$, where  $f^*_i(x^*_i) = 0$ if $x^*_i \leq 1/\beta_i$ and $f^*_i(x^*_i) = x^*_i \beta_i - 2 \sqrt{x^*_i\beta_i} + 1$ if $x^*_i > 1/\beta_i$.
\end{example}

\bibliographystyle{siamplain}
\bibliography{}

\begin{thebibliography}{10}

\bibitem{achdou2017mean}
{\sc Y.~Achdou, M.~Bardi, and M.~Cirant}, {\em Mean field games models of
  segregation}, Mathematical Models and Methods in Applied Sciences, 27 (2017),
  pp.~75--113.

\bibitem{altschuler2020polynomial}
{\sc J.~M. Altschuler and E.~Boix-Adsera}, {\em Polynomial-time algorithms for
  multimarginal optimal transport problems with structure}, Mathematical
  Programming, 199 (2023), pp.~1107--1178.

\bibitem{bauschke2017convex}
{\sc H.~Bauschke and P.~Combettes}, {\em Convex analysis and monotone operator
  theory in {H}ilbert spaces}, Springer, Cham, 2nd~ed., 2017.

\bibitem{beier2023unbalanced}
{\sc F.~Beier, J.~von Lindheim, S.~Neumayer, and G.~Steidl}, {\em Unbalanced
  multi-marginal optimal transport}, Journal of Mathematical Imaging and
  Vision, 65 (2023), pp.~394--413.

\bibitem{benamou2000computational}
{\sc J.-D. Benamou and Y.~Brenier}, {\em A computational fluid mechanics
  solution to the {M}onge-{K}antorovich mass transfer problem}, Numerische
  Mathematik, 84 (2000), pp.~375--393.

\bibitem{benamou2015bregman}
{\sc J.-D. Benamou, G.~Carlier, M.~Cuturi, L.~Nenna, and G.~Peyr{\'e}}, {\em
  Iterative {B}regman projections for regularized transportation problems},
  SIAM Journal on Scientific Computing, 37 (2015), pp.~A1111--A1138.

\bibitem{BenCarDiNen19}
{\sc J.-D. Benamou, G.~Carlier, S.~Di~Marino, and L.~Nenna}, {\em An entropy
  minimization approach to second-order variational mean-field games},
  Mathematical Models and Methods in Applied Sciences, 29 (2019),
  pp.~1553--1583.

\bibitem{bensoussan2013mean}
{\sc A.~Bensoussan, J.~Frehse, and P.~Yam}, {\em Mean field games and mean
  field type control theory}, Springer, New York, NY, 2013.

\bibitem{bensoussan2018mean}
{\sc A.~Bensoussan, T.~Huang, and M.~Lauri{\`e}re}, {\em Mean field control and
  mean field game models with several populations}, Minimax Theory and its
  Applications, 3 (2018), pp.~173--209.

\bibitem{caines2018mean}
{\sc P.~Caines, M.~Huang, and R.~Malham{\'e}}, {\em Mean field games}, in
  Handbook of Dynamic Game Theory, T.~Ba{\c{s}}ar and G.~Zaccour, eds.,
  Springer, Cham, 2018, pp.~345--372.

\bibitem{caluya2021wasserstein}
{\sc K.~F. Caluya and A.~Halder}, {\em Wasserstein proximal algorithms for the
  schr{\"o}dinger bridge problem: Density control with nonlinear drift}, IEEE
  Transactions on Automatic Control, 67 (2021), pp.~1163--1178.

\bibitem{cardaliaguet2015second}
{\sc P.~Cardaliaguet, P.~Graber, A.~Porretta, and D.~Tonon}, {\em Second order
  mean field games with degenerate diffusion and local coupling}, Nonlinear
  Differential Equations and Applications NoDEA, 22 (2015), pp.~1287--1317.

\bibitem{chen2016relation}
{\sc Y.~Chen, T.~Georgiou, and M.~Pavon}, {\em On the relation between optimal
  transport and {S}chr{\"o}dinger bridges: {A} stochastic control viewpoint},
  Journal of Optimization Theory and Applications, 169 (2016), pp.~671--691.

\bibitem{chen2016optimalPartI}
{\sc Y.~Chen, T.~Georgiou, and M.~Pavon}, {\em Optimal steering of a linear
  stochastic system to a final probability distribution, part {I}}, IEEE
  Transactions on Automatic Control, 61 (2016), pp.~1158--1169.

\bibitem{chen2017optimal}
{\sc Y.~Chen, T.~Georgiou, and M.~Pavon}, {\em Optimal transport over a linear
  dynamical system}, IEEE Transactions on Automatic Control, 62 (2017),
  pp.~2137--2152.

\bibitem{chen2018steering}
{\sc Y.~Chen, T.~Georgiou, and M.~Pavon}, {\em Steering the distribution of
  agents in mean-field games system}, Journal of Optimization Theory and
  Applications, 179 (2018), pp.~332--357.

\bibitem{chen2020stochastic}
{\sc Y.~Chen, T.~Georgiou, and M.~Pavon}, {\em Stochastic control liasons:
  {R}ichard {S}inkhorn meets {G}aspard {M}onge on a {S}chr\"odinger bridge},
  SIAM Review, 63 (2021), pp.~249--313.

\bibitem{chen2016optimal}
{\sc Y.~Chen, T.~T. Georgiou, and M.~Pavon}, {\em Optimal steering of a linear
  stochastic system to a final probability distribution, part i}, IEEE
  Transactions on Automatic Control, 61 (2016), pp.~1158--1169.

\bibitem{chen2018state}
{\sc Y.~Chen and J.~Karlsson}, {\em State tracking of linear ensembles via
  optimal mass transport}, IEEE Control Systems Letters, 2 (2018),
  pp.~260--265.

\bibitem{chizat2018scaling}
{\sc L.~Chizat, G.~Peyr{\'e}, B.~Schmitzer, and F.-X. Vialard}, {\em Scaling
  algorithms for unbalanced optimal transport problems}, Mathematics of
  Computation, 87 (2018), pp.~2563--2609.

\bibitem{chizat2018unbalanced}
{\sc L.~Chizat, G.~Peyr{\'e}, B.~Schmitzer, and F.-X. Vialard}, {\em Unbalanced
  optimal transport: Dynamic and kantorovich formulations}, Journal of
  Functional Analysis, 274 (2018), pp.~3090--3123.

\bibitem{cirant2015multi}
{\sc M.~Cirant}, {\em Multi-population mean field games systems with {N}eumann
  boundary conditions}, Journal de Math{\'e}matiques Pures et Appliqu{\'e}es,
  103 (2015), pp.~1294--1315.

\bibitem{como2017resilient}
{\sc G.~Como}, {\em On resilient control of dynamical flow networks}, Annual
  Reviews in Control, 43 (2017), pp.~80--90.

\bibitem{cuturi2013sinkhorn}
{\sc M.~Cuturi}, {\em Sinkhorn distances: {L}ightspeed computation of optimal
  transport}, in Advances in Neural Information Processing Systems (NIPS),
  2013, pp.~2292--2300.

\bibitem{dai1991stochastic}
{\sc P.~Dai~Pra}, {\em A stochastic control approach to reciprocal diffusion
  processes}, Applied mathematics and Optimization, 23 (1991), pp.~313--329.

\bibitem{djehiche2017mean}
{\sc B.~Djehiche, A.~Tcheukam, and H.~Tembine}, {\em Mean-field-type games in
  engineering}, AIMS Electronics and Electrical Engineering, 1 (2017),
  pp.~18--73.

\bibitem{elvander2020multi}
{\sc F.~Elvander, I.~Haasler, A.~Jakobsson, and J.~Karlsson}, {\em
  Multi-marginal optimal transport using partial information with applications
  in robust localization and sensor fusion}, Signal Processing, 171 (2020),
  p.~107474.

\bibitem{fan2022complexity}
{\sc J.~Fan, I.~Haasler, J.~Karlsson, and Y.~Chen}, {\em On the complexity of
  the optimal transport problem with graph-structured cost}, in International
  Conference on Artificial Intelligence and Statistics, PMLR, 2022,
  pp.~9147--9165.

\bibitem{farhangi2010path}
{\sc H.~Farhangi}, {\em The path of the smart grid}, IEEE Power Energy Mag., 8
  (2010).

\bibitem{fleming1975deterministic}
{\sc W.~Fleming and R.~Rishel}, {\em Deterministic and stochastic optimal
  control}, Springer-Verlag, New York, N.Y., 1975.

\bibitem{Follmer88}
{\sc H.~F{\"o}llmer}, {\em Random fields and diffusion processes}, in {\'E}cole
  d'{\'E}t{\'e} de Probabilit{\'e}s de Saint-Flour XV--XVII, 1985--87, P.-L.
  Hennequin, ed., vol.~1362 of Lecture Notes in Mathematics, Springer, Berlin,
  Heidelberg, 1988, pp.~101--203.

\bibitem{gangbo1998optimal}
{\sc W.~Gangbo and A.~{\'S}wi{k{e}}ch}, {\em Optimal maps for the
  multidimensional {M}onge-{K}antorovich problem}, Communications on Pure and
  Applied Mathematics: A Journal Issued by the Courant Institute of
  Mathematical Sciences, 51 (1998), pp.~23--45.

\bibitem{gentil2017analogy}
{\sc I.~Gentil, C.~L{\'e}onard, and L.~Ripani}, {\em About the analogy between
  optimal transport and minimal entropy}, Annales de la Facult{\'e} des
  sciences de Toulouse: Math{\'e}matiques, 26 (2017), pp.~569--600.

\bibitem{georgiou2008metrics}
{\sc T.~T. Georgiou, J.~Karlsson, and M.~S. Takyar}, {\em Metrics for power
  spectra: an axiomatic approach}, IEEE Transactions on Signal Processing, 57
  (2008), pp.~859--867.

\bibitem{haasler2020optimal}
{\sc I.~Haasler, Y.~Chen, and J.~Karlsson}, {\em Optimal steering of ensembles
  with origin-destination constraints}, IEEE Control Systems Letters, 5 (2020),
  pp.~881--886.

\bibitem{haasler2021multimarginal}
{\sc I.~Haasler, A.~Ringh, Y.~Chen, and J.~Karlsson}, {\em Multimarginal
  optimal transport with a tree-structured cost and the {S}chr\"odinger bridge
  problem}, SIAM Journal on Control and Optimization, 59 (2021),
  pp.~2428--2453.

\bibitem{haasler2021scalable}
{\sc I.~Haasler, A.~Ringh, Y.~Chen, and J.~Karlsson}, {\em Scalable computation
  of dynamic flow problems via multimarginal graph-structured optimal
  transport}, Mathematics of Operations Research,  (in press),
  \url{https://doi.org/10.1287/moor.2021.148}.
\newblock arXiv preprint arXiv:2106.14485.

\bibitem{haasler2021control}
{\sc I.~Haasler, A.~Ringh, and J.~Karlsson}, {\em Control and estimation of
  ensembles via strucutred optimal transport: A computational approach based on
  entropy-regularized multimarginal optimal transport}, IEEE Control Systems
  Magazine, 41 (2021), pp.~50--69.

\bibitem{haasler2020multi}
{\sc I.~Haasler, R.~Singh, Q.~Zhang, J.~Karlsson, and Y.~Chen}, {\em
  Multi-marginal optimal transport and probabilistic graphical models}, IEEE
  Transactions on Information Theory, 67 (2021), pp.~4647--4668.

\bibitem{hijab1984asymptotic}
{\sc O.~Hijab}, {\em Asymptotic bayesian estimation of a first order equation
  with small diffusion}, The Annals of Probability, 12 (1984), pp.~890--902.

\bibitem{hindawi2011mass}
{\sc A.~Hindawi, J.-B. Pomet, and L.~Rifford}, {\em Mass transportation with
  {LQ} cost functions}, Acta applicandae mathematicae, 113 (2011),
  pp.~215--229.

\bibitem{huang2012social}
{\sc M.~Huang, P.~Caines, and R.~Malham{\'e}}, {\em Social optima in mean field
  {LQG} control: centralized and decentralized strategies}, IEEE Transactions
  on Automatic Control, 57 (2012), pp.~1736--1751.

\bibitem{huang2006large}
{\sc M.~Huang, R.~Malham{\'e}, and P.~Caines}, {\em Large population stochastic
  dynamic games: closed-loop {McKean}-{Vlasov} systems and the {Nash} certainty
  equivalence principle}, Communications in Information \& Systems, 6 (2006),
  pp.~221--252.

\bibitem{jovanovic1988anonymous}
{\sc B.~Jovanovic and R.~Rosenthal}, {\em Anonymous sequential games}, Journal
  of Mathematical Economics, 17 (1988), pp.~77--87.

\bibitem{karlsson2017generalized}
{\sc J.~Karlsson and A.~Ringh}, {\em Generalized {S}inkhorn iterations for
  regularizing inverse problems using optimal mass transport}, SIAM Journal on
  Imaging Sciences, 10 (2017), pp.~1935--1962.

\bibitem{krishnan2018distributed}
{\sc V.~Krishnan and S.~Mart{\'\i}nez}, {\em Distributed optimal transport for
  the deployment of swarms}, in 2018 IEEE Conference on Decision and Control
  (CDC), IEEE, 2018, pp.~4583--4588.

\bibitem{lachapelle2011mean}
{\sc A.~Lachapelle and M.-T. Wolfram}, {\em On a mean field game approach
  modeling congestion and aversion in pedestrian crowds}, Transportation
  research part B: methodological, 45 (2011), pp.~1572--1589.

\bibitem{lamond1981bregman}
{\sc B.~Lamond and N.~Stewart}, {\em Bregman's balancing method},
  Transportation Research Part B: Methodological, 15 (1981), pp.~239--248.

\bibitem{lasry2007mean}
{\sc J.-M. Lasry and P.-L. Lions}, {\em Mean field games}, Japanese journal of
  mathematics, 2 (2007), pp.~229--260.

\bibitem{lee2021controlling}
{\sc W.~Lee, S.~Liu, H.~Tembine, W.~Li, and S.~Osher}, {\em Controlling
  propagation of epidemics via mean-field control}, SIAM Journal on Applied
  Mathematics, 81 (2021), pp.~190--207.

\bibitem{leonard2012schrodinger}
{\sc C.~L{\'e}onard}, {\em From the {S}chr{\"o}dinger problem to the
  {M}onge--{K}antorovich problem}, Journal of Functional Analysis, 262 (2012),
  pp.~1879--1920.

\bibitem{leonard2013schrodinger}
{\sc C.~L{\'e}onard}, {\em A survey of the {S}chr{\"o}dinger problem and some
  of its connections with optimal transport}, Discrete \& Continuous Dynamical
  Systems - A, 34 (2014), pp.~1533--1574.

\bibitem{liero2018optimal}
{\sc M.~Liero, A.~Mielke, and G.~Savar{\'e}}, {\em Optimal entropy-transport
  problems and a new hellinger--kantorovich distance between positive
  measures}, Inventiones mathematicae, 211 (2018), pp.~969--1117.

\bibitem{lin2022complexity}
{\sc T.~Lin, N.~Ho, M.~Cuturi, and M.~Jordan}, {\em On the complexity of
  approximating multimarginal optimal transport}, Journal of Machine Learning
  Research, 23 (2022), pp.~1--43.

\bibitem{luo1993convergence}
{\sc Z.-Q. Luo and P.~Tseng}, {\em On the convergence rate of dual ascent
  methods for linearly constrained convex minimization}, Mathematics of
  Operations Research, 18 (1993), pp.~846--867.

\bibitem{meyer2014road}
{\sc G.~Meyer and S.~Beiker}, eds., {\em Road Vehicle Automation}, Springer,
  Cham, 2014.

\bibitem{nenna2016numerical}
{\sc L.~Nenna}, {\em Numerical methods for multi-marginal optimal
  transportation}, PhD thesis, PSL Research University, 2016.

\bibitem{nocedal2006numerical}
{\sc J.~Nocedal and S.~Wright}, {\em Numerical optimization}, Springer, New
  York, NY, 2nd~ed., 2006.

\bibitem{ortega1970iterative}
{\sc J.~Ortega and W.~Rheinboldt}, {\em Iterative solution of nonlinear
  equations in several variables}, Academic Press, New York, NY, 1970.

\bibitem{pass2015multi}
{\sc B.~Pass}, {\em Multi-marginal optimal transport: theory and applications},
  ESAIM: Mathematical Modelling and Numerical Analysis, 49 (2015),
  pp.~1771--1790.

\bibitem{peyre2019computational}
{\sc G.~Peyr{\'e} and M.~Cuturi}, {\em Computational optimal transport: With
  applications to data science}, Foundations and Trends{\textregistered} in
  Machine Learning, 11 (2019), pp.~355--607.

\bibitem{piccoli2014generalized}
{\sc B.~Piccoli and F.~Rossi}, {\em Generalized wasserstein distance and its
  application to transport equations with source}, Archive for Rational
  Mechanics and Analysis, 211 (2014), pp.~335--358.

\bibitem{astrom1970introduction}
{\sc K.~\r{A}str\"{o}m}, {\em Introduction to stochastic control theory},
  Dover, Mineola, NY, 2006.
\newblock Unabridged republication of original published by Academic Press,
  1970.

\bibitem{ringh2021efficient}
{\sc A.~Ringh, I.~Haasler, Y.~Chen, and J.~Karlsson}, {\em Efficient
  computations of multi-species mean field games via graph-structured optimal
  transport}, in 2021 60th IEEE Conference on Decision and Control (CDC), IEEE,
  2021, pp.~5261--5268.

\bibitem{rockafellar1970convex}
{\sc R.~Rockafellar}, {\em Convex analysis}, Princeton Mathematical Series,
  Princeton University Press, Princeton, NJ, 1970.

\bibitem{ruschendorf1995optimal}
{\sc L.~R{\"u}schendorf}, {\em Optimal solutions of multivariate coupling
  problems}, Applicationes Mathematicae, 23 (1995), pp.~325--338.

\bibitem{ruschendorf2002n}
{\sc L.~R{\"u}schendorf and L.~Uckelmann}, {\em On the n-coupling problem},
  Journal of multivariate analysis, 81 (2002), pp.~242--258.

\bibitem{singh2020inference}
{\sc R.~Singh, I.~Haasler, Q.~Zhang, J.~Karlsson, and Y.~Chen}, {\em Inference
  with aggregate data in probabilistic graphical models: An optimal transport
  approach}, IEEE Transactions on Automatic Control, 67 (2022), pp.~4483--4497.

\bibitem{sinigaglia2021optimal}
{\sc C.~Sinigaglia, S.~Bandyopadhyay, M.~Quadrelli, and F.~Braghin}, {\em
  Optimal-transport-based control of particle swarms for orbiting rainbows
  concept}, Journal of Guidance, Control, and Dynamics, 44 (2021),
  pp.~2108--2117.

\bibitem{teh2002propagation}
{\sc Y.~Teh and M.~Welling}, {\em The unified propagation and scaling
  algorithm}, Advances in neural information processing systems,  (2002),
  pp.~953--960.

\bibitem{todorov2008general}
{\sc E.~Todorov}, {\em General duality between optimal control and estimation},
  in 2008 47th IEEE Conference on Decision and Control, IEEE, 2008,
  pp.~4286--4292.

\bibitem{tseng1993dual}
{\sc P.~Tseng}, {\em Dual coordinate ascent methods for non-strictly convex
  minimization}, Mathematical programming, 59 (1993), pp.~231--247.

\bibitem{villani2003topics}
{\sc C.~Villani}, {\em Topics in optimal transportation}, American Mathematical
  Society, Providence, RI, 2003.

\bibitem{watts1998collective}
{\sc D.~J. Watts and S.~H. Strogatz}, {\em Collective dynamics of 'small-world'
  networks}, Nature, 393 (1998), p.~440.

\end{thebibliography}

\end{document}